\documentclass[12pt]{amsart}
\usepackage{graphicx}
\usepackage{color}
\allowdisplaybreaks
\usepackage{amsmath,amssymb,mathrsfs,amsthm}

\usepackage{verbatim}
\usepackage[spanish,USenglish]{babel} 
\usepackage{color}
\usepackage{tikz}
\usepackage{graphicx}

\newtheorem{thm}{Theorem}[section]

\newtheorem{proposition}[thm]{Proposition}
\newtheorem{theorem}[thm]{Theorem}
\newtheorem{lemma}[thm]{Lemma}

\theoremstyle{definition}
\newtheorem{definition}[thm]{Definition}
\newtheorem{example}[thm]{Example}
\theoremstyle{remark}
\newtheorem{remark}[thm]{Remark}
\numberwithin{equation}{section}

\newcommand\Real{{\mathfrak R}{\mathfrak e}\,} %
\newcommand\Imag{{\mathfrak I}{\mathfrak m}\,} %

\newcommand{\R}{\mathbb{{R}}}

\newcommand{\N}{\mathbb{{N}}}

\newcommand{\C}{\mathbb{{C}}}
\newcommand{\D}{\mathbb{{D}}}

\begin{document}

\title[Generalized Ces\`aro operators on Sobolev-Lebesgue sequence spaces]
{Generalized Ces\`aro operators, fractional finite differences and Gamma functions }

\author[Luciano Abadias]{Luciano Abadias}
\address[L. Abadias]{Centro Universitario de la Defensa, Instituto Universitario de Matem\'aticas y Aplicaciones, 50090 Zaragoza, Spain.}
\email{labadias@unizar.es}

\author{Pedro J. Miana}
\address[P. J. Miana]{Departamento de Matem\'aticas, Instituto Universitario de Matem\'aticas y Aplicaciones, Universidad de Zaragoza, 50009 Zaragoza, Spain.}
\email{pjmiana@unizar.es}

\thanks{Authors  have been partially supported by Project MTM2016-77710-P, DGI-FEDER, of the MCYTS and Project E-64, D.G. Arag\'on, Universidad de Zaragoza, Spain.}

\subjclass[2010]{47B38, 40G05, 47A10, 33B15}

\keywords{Ces\`{a}ro operators, Gamma functions, Ces\`{a}ro numbers, Lebesgue sequence spaces, spectrum set}


\begin{abstract} In this paper, we present a complete spectral research of generalized Ces\`aro operators on Sobolev-Lebesgue sequence spaces. The main idea is to subordinate  such operators to suitable $C_0$-semigroups on these sequence spaces. We introduce that family of sequence spaces using  the fractional finite differences and we prove some structural properties similar to classical Lebesgue sequence spaces. In order to show the main results  about fractional finite differences, we state equalities involving sums of quotients of  Euler's Gamma functions. Finally, we display some graphical representations of the spectra of generalized Ces\`aro operators.

\end{abstract}

\maketitle

\section*{Introduction}\label{intro}

Given a sequence  $f=(f(n))_{n=0}^\infty$ of complex numbers, we consider its sequence of averages
$$
{\mathcal C}f(n):={1\over n+1}\sum_{j=0}^n f(j), \qquad n \in \N\cup\{0\},
$$
often called Ces\`{a}ro means. They were introduced by E. Ces\`{a}ro in 1890, see for example \cite{Cesaro}, and they have several applications, e.g. to (originally) the multiplication of series, in the theory of Fourier series or in  asymptotic analysis. Moreover, Ces\`{a}ro means of natural order $m$, ${\mathcal C}_mf$, $({\mathcal C}_1={\mathcal C})$, where
$$
{\mathcal C}_mf(n):={m (n!)\over (n+m)!}\sum_{j=0}^n {(n-j+m-1)!\over (n-j)!} f(j), \qquad n \in \N\cup\{0\},
$$
were also introduced in \cite{Cesaro}.

In 1965 the (discrete) Ces\`{a}ro operator, $f \mapsto {\mathcal C}(f)$ from $\ell^2$ to $\ell^2,$  was originally introduced in \cite{Brown}. Authors showed that $ {\mathcal C}$ is a bounded operator, $\Vert {\mathcal C}\Vert=2$,  and the spectrum of  $ {\mathcal C}$ is the closed disc $\{\lambda\,:\vert \lambda-1\vert \le 1\}$ (\cite[Theorem 2]{Brown}). In \cite{Leibo}, authors proved that $\mathcal{C}$ is a bounded linear operator from $\ell^p$ into itself  and  its spectrum is the closed ball centered in $q/2$ and radius $q/2,$ for $1< p\leq \infty$ and ${1\over p}+{1\over q}=1$.

The boundedness of the Ces\`{a}ro operator ${\mathcal C}$  from $\ell^p$ into itself is a straightforward consequence of the well-known Hardy inequality, see a nice survey about this inequality in \cite{KMP}. The following generalized Hardy inequality, see \cite[Theorem 3.18, p.227]{HLP}, allows to show the boundedness of Ces\`{a}ro operators ${\mathcal C}_m$   from $\ell^p$ into itself: Let $\beta>0$ and $1<p<\infty,$ then \begin{equation*}\label{hardy}
\biggl(\sum_{n=1}^{\infty}\left|\frac{1}{n^{\beta}} \sum_{j=0}^n (n-j)^{\beta-1}f(j)\right|^p\biggr)^{1/p}\leq \frac{\Gamma(\beta)\Gamma(1-1/p)}{\Gamma(\beta+1-1/p)}\biggl(\sum_{n=1}^{\infty}|f(n)|^p\biggr)^{1/p},\quad f\in\ell^p.
\end{equation*}
Also, the dual inequality holds, \begin{equation}\label{hardyDual}
\biggl(\sum_{n=1}^{\infty}\left| \sum_{j=n}^{\infty}\frac{(j-n)^{\beta-1}}{j^{\beta}} f(j)\right|^p\biggr)^{1/p}\leq \frac{\Gamma(\beta)\Gamma(1/p)}{\Gamma(\beta+1/p)}\biggl(\sum_{n=1}^{\infty}|f(n)|^p\biggr)^{1/p},\quad f\in\ell^p.
\end{equation}




For any complex number $\alpha,$ we denote by \begin{equation*}\label{definition} k^{\alpha}(n):=\frac{\alpha(\alpha+1)\cdots(\alpha+n-1)}{n!}\quad \text{for }n\in\N,\, k^{\alpha}(0):=1,\end{equation*} the known Ces\`aro numbers which are studied deeply in \cite[Vol. I, p.77]{Zygmund} and denoted by $A_{n}^{\alpha-1}$. The kernels $k^{\alpha}$ have played a key role in results about operator theory and fractional difference equations, see \cite{A, ADT, ALMV, Lizama}. Note that the sequence $k^\alpha$ can be written as $$k^{\alpha}(n)=\frac{\Gamma(n+\alpha)}{\Gamma(\alpha)\Gamma(n+1)}={n+\alpha-1\choose \alpha-1}=(-1)^n\binom{-\alpha}{n},\qquad n\in\N_{0},\,\alpha\in \C\backslash\{0,-1,-2,\ldots\},$$
where $\Gamma$ is the  Euler Gamma function. Also, the kernel $(k^{\alpha}(n))_{n\in\N_0}$ could be defined by the generating function, that is, \begin{equation}\label{generating}
\displaystyle\sum_{n=0}^{\infty}k^{\alpha}(n)z^n=\frac{1}{(1-z)^{\alpha}},\quad |z|<1.
 \end{equation} Therefore, these kernels satisfy the semigroup property, $k^{\alpha}*k^{\beta}=k^{\alpha+\beta}$ for $\alpha, \beta \in \C$.

In this paper, we study the generalized Ces\`{a}ro operators of order $\beta$ (in some paper called $\beta$-Ces\`{a}ro operators, \cite{Stevik, Xiao}) and its dual operator of order $\beta$,  $\mathcal{C}_{\beta} $ and $\mathcal{C}^*_{\beta}$ respectively, given by \begin{eqnarray*}
\mathcal{C}_{\beta}f(n)&:=&\frac{1}{k^{\beta+1}(n)}\sum_{j=0}^nk^{\beta}(n-j)f(j), \cr \mathcal{C}^*_{\beta}f(n)&:=&\sum_{j=n}^{\infty}\frac{1}{k^{\beta+1}(j)}k^{\beta}(j-n)f(j),
\end{eqnarray*}
for $n\in \N_0$ and $\Real\beta >0,$ acting on a new family of Sobolev-Lebesgue  spaces $\tau_p^\alpha,$ for $\alpha\geq 0,$ introduced in section \ref{sect2}. These families of sequence space  include the classical Lebesgue spaces $\ell^p $ for $\alpha=0$ and $1\le p\le \infty$, i.e, $\ell^p =\tau_p^0$. The Banach algebras $\tau_1^\alpha$ were introduced in \cite{ALMV} to define bounded algebra homomorphisms linked to $(C, \alpha)$-bounded operators. In \cite[Theorem 3.1 and Remark 3.3]{Stempak}, it is shown that the generalized Ces\`{a}ro operators $\mathcal{C}_{\beta} $ and $\mathcal{C}^*_{\beta}$ are bounded on $\ell^p$ for $1<p<\infty$; $p=\infty$ for $\mathcal{C}_{\beta} $ and $p=1$ for  $\mathcal{C}^*_{\beta}$.

The main idea of the paper is to represent the operators $\mathcal{C}_{\beta} $ and $\mathcal{C}^*_{\beta}$ via $C_0$-semigroups of operators. In particular, we consider the $C_0$-semigroups $(T_p(t))_{t\geq 0}$ and $(S_p(t))_{t\geq 0}$ acting on $\tau^{\alpha}_p$ (section \ref{sect5}), given by \begin{eqnarray*}
T_p(t)f(n)&:=&e^{-\frac{t}{p}}\displaystyle\sum_{j=0}^n\binom{n}{j}e^{-tj}(1-e^{-t})^{n-j}f(j),\cr S_p(t)f(n)&:=&e^{-t(1-\frac{1}{p}+n)}\displaystyle\sum_{j=n}^{\infty}\binom{j}{n}(1-e^{-t})^{j-n}f(j),
\end{eqnarray*}
for   $n\in\N_0$ and $t\geq 0,$  and we will get
 \begin{eqnarray*}
 \mathcal{C}_{\beta}f(n)&=&\displaystyle\beta\int_0^{\infty}(1-e^{-t})^{\beta-1}e^{-t(1-\frac{1}{p})}T_p(t)f(n)\,dt,\quad n\in\N_0,\, 1<p\leq\infty,\cr
 \mathcal{C}^*_{\beta}f(n)&=&\beta\int_0^{\infty}(1-e^{-t})^{\beta-1}e^{-\frac{t}{p}}S_p(t)f(n)\,dt,\quad n\in\N_0,\,\,\, 1\leq p <\infty,
\end{eqnarray*}
for $\Real\beta >0$, see Theorem \ref{bounded}. These two equalities give additional information about the connection between one-parameter semigroups and the operator $\mathcal{C}_{\beta},$ a question posed in \cite[Section 3]{Stempak}.

The powerful theory of $C_0$-semigroups of operators allows us to get bounds for norms of operators (see proof of Theorem \ref{bounded}), and  to describe the spectra of $\mathcal{C}_{\beta} $ and $\mathcal{C}^*_{\beta}$ in Theorem \ref{spectro} via a spectral mapping theorem for sectorial operators and functions which belong to the extended Dunford-Riesz class, see \cite[Chapter 2]{Haase}. Note that these $C_0$-semigroups, $(T_p(t))_{t\geq 0}$ and $(S_p(t))_{t\geq 0},$ are not holomorphic (see Remark \ref{holomor}), and the classical spectral mapping theorem for holomorphic semigroups is  not applicable. In the last section, we present some pictures of the spectrum of $\mathcal{C}_{\beta}$ for some particular $\beta$ complex numbers.

The  norms $\|\quad\|_{\alpha,p}$ on the Sobolev-Lebesgue spaces $\tau_p^\alpha$ are defined involving fractional finite (or Weyl) differences $W^\alpha$ and the kernels $(k^\alpha (n))_{n\in \N_0}$, i.e.,
$$\|f\|_{\alpha,p}:=\left(\sum_{n=0}^\infty |W^\alpha f(n)|^p(k^{\alpha+1}(n))^p\right)^{1/p}, \qquad $$ whenever this expression is finite for $\alpha \ge 0$ and $1\le p\le \infty$, see Definition \ref{taualphasp}. The spaces $\tau_p^\alpha$ are module respect  to the Banach algebras $\tau_1^\alpha$ and  the usual convolution product $\ast$ on sequence spaces (Theorem \ref{propiedades2}). Note that the spaces $\tau_p^\alpha$ form scales on $p$ and on $\alpha$, see Theorem \ref{propiedades}.

We consider the finite difference $W$ defined by  $Wf(n):=f(n)-f(n+1)$ and the iterative finite difference $W^m$ given by $W^mf:=W^{m-1}(Wf)$ for $m\in \N$. For $\alpha \in \R\backslash\N,$ the Weyl difference $W^\alpha$ allows to obtain fractional finite difference  of order  $\alpha >0$ and
$$W^{\alpha}f(n)=\displaystyle\sum_{j=n}^{\infty} k^{-\alpha}(j-n)f(j),\qquad n\in\N_0,$$ whenever both expressions  make sense.
This formula shows  how  deep the connection between $W^\alpha$ and kernels $k^\alpha$ is, see section \ref{Weyl} and Theorem \ref{quivalent}.

Different theories of fractional differences have been developed in recent years with interesting applications to boundary value problems and concrete models coming from biological problems, see for example \cite{AtSe10} and \cite{Go12}. Note that the Weyl fractional difference $W^\alpha$ coincides or is closed to other fractional sums presented in \cite[Section 1]{AtEl09} or \cite[Theorem 2.5]{Lizama1}.

We also give the following formulae which show  the natural behavior between $W^\alpha$ and the $C_0$-semigroups $(T_p(t))_{t\geq 0}$ and $(S_p(t))_{t\geq 0},$
  \begin{eqnarray*}
  W^{\alpha}T_p(t)f(n)&=&e^{-t\alpha}T_p(t)W^{\alpha}f(n), \cr
  k^{\alpha+1}(n)W^{\alpha}S_p(t)f(n)&=&S_p(t)(k^{\alpha +1}W^{\alpha}f)(n),
  \end{eqnarray*}
for $\alpha>0$, $t\geq 0,$ and $n\in\N_0,$ as consequences of  Lemmata \ref{WeylSem} and \ref{WeylSemT}. To obtain these technical lemmata, fine calculations which involve sums of kernels $k^\alpha$ (see Theorem \ref{llave}) or quotients of Gamma functions are needed. In particular, in Theorem  \ref{funda}, we show the following nice identity which seems to be unknown until now,
$$
{\Gamma(m+r+1)\over \Gamma(v+m+r) }\sum_{l=0}^m{\Gamma(l-\alpha)\over \Gamma(l+1)}{\Gamma(v+\alpha+r+m-l)\over \Gamma(r+1+m-l)}=
 {\Gamma(\alpha+r+v)\over \Gamma(r) }\sum_{l=0}^m{\Gamma(l-\alpha)\over \Gamma(l+1)}{\Gamma(l+r)\over \Gamma(v+l+r)}.
$$
for $v,r>0$, $\alpha\in \R_+\backslash\{0,1,2,\ldots\},$ and $m\in \N\cup\{0\}$.

Ces\`{a}ro operators and generalized Ces\`{a}ro operators have been treated in several spaces and with different techniques. In the bibliography, we present only a small part of this literature. Applications of $C_0$-semigroups of operators in the study of Ces\`{a}ro operators was initialed  by C.C. Cowen in \cite{Cowen}. Ces\`{a}ro operators on the Hardy space $H^p$ defined on the unit disc are considered in \cite{PLMSSis} and on the Bergman space $A^p(\D)$ in \cite{ArchivSis}.  However, as K. Stempak pointed out in \cite[Section 3]{Stempak}, the generalized Ces\`{a}ro operator is natural to consider (or even  more natural than on the Hardy spaces $H^p$) on  the sequence spaces $\ell^p$, and it could be interesting to study spectral properties of ${\mathcal C}_\beta$ from the point of view of operator theory and semigroups. Our research completes several  previous  known results which have been obtained on the unit disc, see Remarks \ref{compos} and \ref{holo}.

Recently, in \cite{arvanti}, Ces\`{a}ro operators and semigroups on the $H^p$ space defined on the half-plane are treated, and also on the Lebesgue-Sobolev spaces defined on the half-line $\R_+$ and on the whole line $\R,$ see \cite{LMPL}. The continuous case presented in \cite{LMPL} is easier than the discrete case, which we present in this manuscript. Although the philosophy of both papers is parallel, many more technical difficulties  arise in sequence spaces, see  for example sections \ref{sect0} and \ref{sect1}, and Theorem \ref{spectro}.

\medskip
\noindent{\bf Notation.} We write as $\N$ the set of natural numbers and $\N_0:=\N \cup\{0\}$; $\R$ is the set of real numbers and $\R_+$ is the set of non-negative real numbers; $\C$ is the set of complex numbers, $\C_+$ and $\C_-$ the set of complex numbers with positive and negative real part respectively and $\D:=\{z\in \C\,\,:\,\,\vert z\vert <1\}.$

The sequence Lebesgue space $\ell^p$ is the set of complex sequences $f=(f(n))_{n\in\N_0}$ such that
$$\displaystyle\Vert f\Vert_p:=\left(\sum_{n=0}^{\infty}\lvert f(n)\rvert^p\right)^{1\over p}<\infty,$$ for $1\le p<\infty,$
$\ell^\infty$ the set of bounded complex sequence with the suprem norm and $c_{0,0}$ the set of vector-valued sequences with finite support. The usual convolution product $\ast$ of sequences $f=(f(n))_{n\in\N_0}$ and $g=(g(n))_{n\in\N_0}$   is defined by
$$
f\ast g(n)=\sum_{j=0}^n f(n-j) g(j), \qquad n\in \N_0.
$$
Note that $(\ell^1, \ast)$ is a Banach algebra and $(\ell^p, \ast)$ is a Banach module of $\ell^1$ for $1<p\le \infty$.

We denote by $\Gamma$ and $\bf{B}$ the usual  Euler  Gamma  and Beta functions given by
 \begin{eqnarray}\label{beta}
  \Gamma (z)&=&\int_0^\infty e^{-t}t^{z-1}dt, \quad \Real z>0,\cr
 {\bf B}(u,v)&=&{\Gamma(u)\Gamma(v)\over\Gamma(u+v)}= \int_0^\infty e^{-tv}(1-e^{-t})^{u-1}dt, \qquad \Real u, \Real v>0. \label{beta}
 \end{eqnarray}
and $\Gamma$ is extended to $\C \backslash \{0,-1,-2,\ldots\}$ using the identity $\Gamma(z+1)=z\Gamma(z)$.

\section{Sums of quotients of Gamma functions}\label{sect0}

In this first section, we display technical lemmas involving Euler's Gamma functions. They are used in some main results along the paper and they are interesting by themselves.

\begin{lemma} \label{fits} For $v, u>0$, and $\alpha\in \R_+\backslash\{0,1,2,\ldots\},$ the following equality holds
$$
\sum_{l=0}^\infty{\Gamma(l-\alpha)\over \Gamma(l+1)}{\Gamma(u+l+1)\over \Gamma(v+u+l+1)}= {\Gamma(u+1)\Gamma(v+\alpha)\Gamma(-\alpha)\over \Gamma(v)\Gamma(u+\alpha+v+1)}.
$$

\end{lemma}
\begin{proof} By the generating formula (\ref{generating}), we get
\begin{eqnarray*}
&\quad&{\Gamma(v)\over \Gamma(-\alpha)}\sum_{l=0}^\infty{\Gamma(l-\alpha)\over \Gamma(l+1)}{\Gamma(u+l+1)\over \Gamma(v+u+l+1)}=\sum_{l=0}^\infty k^{-\alpha}(l){\Gamma(v)\Gamma(u+l+1)\over \Gamma(v+u+l+1)}\cr
&\quad&=\displaystyle\sum_{n=0}^{\infty}k^{-\alpha}(l)\int_0^1(1-x)^{v-1}x^{u+l}dx=
\int_0^1(1-x)^{v-1}x^{u}\sum_{n=0}^{\infty}k^{-\alpha}(l)x^ldx\cr
&\quad&=\int_0^1(1-x)^{v+\alpha-1}x^{u}dx={\Gamma(v+\alpha)\Gamma(u+1)\over \Gamma(u+\alpha+v+1)},\cr
\end{eqnarray*}
and the proof is completed.
\end{proof}

\begin{lemma}\label{key}
For $v,r>0$, $\alpha\in \R_+\backslash\{0,1,2,\ldots\},$ and $m\in \N\cup\{0\}$, the following equality holds
$$
\sum_{l=0}^m\left(\alpha r+l(v+\alpha)\right){\Gamma(l-\alpha)\over \Gamma(l+1)}{\Gamma(l+r)\over \Gamma(v+1+l+r)}=-{\Gamma(m+1+r)\over \Gamma(m+1)}
{\Gamma(m+1-\alpha)\over \Gamma(v+m+r+1)}.
$$

\end{lemma}
\begin{proof} We prove the equality by induction on $m$. For $m=0$, the equality holds trivially. By induction's hypothesis, it is enough to check that
\begin{eqnarray*}
&\,&(\alpha r+(m+1)(v+\alpha){\Gamma(m+1-\alpha)\over \Gamma(m+2)}{\Gamma(m+1+r)\over \Gamma(v+m+2+r)}-{\Gamma(m+1+r)\over \Gamma(m+1)}
{\Gamma(m+1-\alpha)\over \Gamma(v+m+r+1)}\cr
&\,&= {\Gamma(m+1+r)\over \Gamma(m+1)}
{\Gamma(m+1-\alpha)\over \Gamma(v+m+r+1)}\left({(\alpha r+(m+1)(v+\alpha)\over(m+1)(v+m+1+r)}-1\right)\cr
&\,&= -{\Gamma(m+1+r)\over \Gamma(m+1)}
{\Gamma(m+1-\alpha)\over \Gamma(v+m+r+1)}\left({(m+1-\alpha)(m+1+r)\over(m+1)(v+m+1+r)}\right)\cr
&\,&=-{\Gamma(m+2+r)\over \Gamma(m+2)}
{\Gamma(m+2-\alpha)\over \Gamma(v+m+r+2)}
\end{eqnarray*}
and we conclude the proof.
\end{proof}

\begin{theorem}\label{funda} For $v,r>0$, $\alpha\in \R_+\backslash\{0,1,2,\ldots\},$ and $m\in \N\cup\{0\}$, the following equality holds
$$
{\Gamma(m+r+1)\over \Gamma(v+m+r) }\sum_{l=0}^m{\Gamma(l-\alpha)\over \Gamma(l+1)}{\Gamma(v+\alpha+r+m-l)\over \Gamma(r+1+m-l)}=
 {\Gamma(\alpha+r+v)\over \Gamma(r) }\sum_{l=0}^m{\Gamma(l-\alpha)\over \Gamma(l+1)}{\Gamma(l+r)\over \Gamma(v+l+r)}.
$$
\end{theorem}

\begin{proof} We prove the equality by induction on $m$. For $m=0$, the equality holds trivially. Now we define functions  $G_{m, r}$ and $F_{m, r}$
\begin{eqnarray*}
G_{m, r}(v)&:=&{\Gamma(m+r+1)\over \Gamma(v+m+r) }\sum_{l=0}^m{\Gamma(l-\alpha)\over \Gamma(l+1)}{\Gamma(v+\alpha+r+m-l)\over \Gamma(r+1+m-l)},\cr
F_{m,r}(v)&:=&  {\Gamma(\alpha+r+v)\over \Gamma(r) }\sum_{l=0}^m{\Gamma(l-\alpha)\over \Gamma(l+1)}{\Gamma(l+r)\over \Gamma(v+l+r)},
\end{eqnarray*}
 for $v>0$ and we shall conclude that $G_{m, r}=F_{m, r}$ for $m\in \N$ and $r>0$.  Note that
\begin{eqnarray*}
G_{m+1, r}(v)&:=& G_{m,r+1}(v)+ g_{m,r}(v),\cr
F_{m+1,r}(v)&:=&  F_{m,r}(v)+ f_{m,r}(v),
\end{eqnarray*}
where
\begin{eqnarray*}
g_{m,r}(v)&=&{\Gamma(m+r+2)\over \Gamma(v+m+r+1) }{\Gamma(v+\alpha+r)\over \Gamma(r+1) }{\Gamma(m+1-\alpha)\over \Gamma(m+2) },\cr
f_{m,r}(v)&=&{\Gamma(v+\alpha+r)\over \Gamma(r) }{\Gamma(m+1+r)\over \Gamma(v+m+r+1) }{\Gamma(m+1-\alpha)\over \Gamma(m+2) },\cr
\end{eqnarray*}
for $v>0$. By induction's hypothesis we have that
\begin{eqnarray*}
G_{m+1,r}(v)&=& G_{m,r+1}(v)+g_{m,r}(v)=F_{m,r+1}(v)+g_{m,r}(v)\cr
&=&F_{m+1, r+1}-f_{m,r+1}(v)+g_{m,r}(v),
\end{eqnarray*}
and to finish the proof it is enough to conclude that
$$
F_{m+1,r}(v)-F_{m+1, r+1}(v)=-f_{m,r+1}(v)+g_{m,r}(v).
$$

Now, in one hand, we have that
$$
g_{m,r}(v)-f_{m,r+1}(v)={\Gamma(v+\alpha+r)\over \Gamma(r+1) }{\Gamma(m+2+r)\over \Gamma(v+m+r+2) }{\Gamma(m+2-\alpha)\over \Gamma(m+2) }.
$$
In other hand, we have that
$$
F_{m+1,r}(v)-F_{m+1, r+1}(v)=-{\Gamma(\alpha+r+v)\over \Gamma(r+1) }\sum_{l=0}^{m+1}{\Gamma(l-\alpha)\over \Gamma(l+1)}{\Gamma(l+r)\over \Gamma(v+l+r+1)}\left(\alpha r+ l(v+\alpha)\right)
$$
and we apply Lemma \ref{key} to obtain
$$F_{m+1,r}(v)-F_{m+1, r+1}(v)=
 {\Gamma(\alpha+r+v)\Gamma(m+2+r)\Gamma(m+2-\alpha)\over\Gamma(v+m+r+2) \Gamma(m+2)\Gamma(r+1) }=g_{m,r}(v)-f_{m,r+1}(v)
$$
proving the result.
\end{proof}

To conclude this section we obtain some particular cases in Theorem \ref{funda} for $r\in \{1,2\}$. We omit the proof which the reader may  make by induction on $m$.

\begin{theorem} Take $v>0$, $\alpha\in \R_+\backslash\{0,1,2,\ldots\}$ and $m\in \N\cup\{0\}$. Then
\begin{eqnarray*}
\sum_{l=0}^m{\Gamma(l-\alpha)\over \Gamma(v+l+1)}&=&{\Gamma(-\alpha)\over (\alpha +v)\Gamma(v)}-{\Gamma(m+1-\alpha)\over (\alpha+v)\Gamma(v+m+1)}.\cr
\sum_{l=0}^m(l+1){\Gamma(l-\alpha)\over \Gamma(v+l+1)}&=&{ \Gamma(-\alpha)\over (\alpha +v+1)(\alpha +v)\Gamma(v)}\cr
&\,&\qquad \qquad-{(\alpha m+mv+\alpha+m+2v+1)\Gamma(m+1-\alpha)\over (\alpha+v+1)(\alpha+v)\Gamma(v+m+2)}.
\end{eqnarray*}

\end{theorem}

\begin{remark} For $r\in \N$, $v>0$, $m\in \N\cup\{0\}$ and $0<\alpha<1$ we conjecture that
\begin{eqnarray*}
{\Gamma(v+\alpha+r)\over \Gamma(r)}\sum_{l=0}^m{\Gamma(l+r)\Gamma(l-\alpha)\over \Gamma(v+l+r)\Gamma(l+1)}={\Gamma(v+\alpha)}\left({\Gamma(-\alpha)\over\Gamma(v)}-{P_r(\alpha,m,v)\Gamma(m-\alpha+1)\over\Gamma(m+v+r)}\right)
\end{eqnarray*}
where $P_r(\alpha,m,v)$ is a polynomial of degree $r-1$ in the variables $\alpha, m$ and $v$.

\end{remark}

\section{New properties of Ces\`aro numbers }\label{sect1}

As a function of $n,$ $k^\alpha$ is increasing  for $\alpha >1$, decreasing for $1>\alpha >0$ and $k^1(n)=1$ for $n\in \N$ (\cite[Theorem III.1.17]{Zygmund}). Furthermore, it is straightforward to check that  $k^\alpha(n)\le k^\beta(n)$ for $\beta \ge \alpha>0$ and $n\in \N_0$.

In the following, we will use the asymptotic behaviour of the sequences $k^{\alpha}.$ Note that for $\alpha\in\{0,-1,-2,\ldots\}$, $k^{\alpha}(n)=0$ for $n> -\alpha.$ In addition, for real $\alpha\notin\{0,-1,-2,\ldots\},$   \begin{equation}\label{double}
 k^{\alpha}(n)=\frac{n^{\alpha-1}}{\Gamma(\alpha)}(1+O({1\over n})), \qquad n\in \N. \end{equation}
(\cite[Vol. I, p.77 (1.18)]{Zygmund}). Moreover this property holds in a more general context. Let $\alpha,z\in\C,$ then $$\frac{\Gamma(z+\alpha)}{\Gamma(z)}=z^{\alpha}(1+\frac{\alpha(\alpha+1)}{2z}+O(|z|^{-2})),\quad |z|\to\infty,$$ whenever $z\neq 0,-1,-2,\ldots$ and $z\neq -\alpha,-\alpha-1,\ldots,$ see \cite{ET}. We are interested in the two following particular cases,

\begin{equation}\label{double2}
k^{\alpha}(n)=\frac{n^{\alpha-1}}{\Gamma(\alpha)}(1+O({1\over n})), \quad n\in \N,\,\alpha\notin\{0,-1,-2,\ldots\},
\end{equation}

\begin{equation}\label{double3}
\frac{\Gamma(z+\alpha)}{\Gamma(z)}=z^{\alpha}(1+O({1\over |z|})), \quad z\in\C_+,\,\Real\alpha>0.
\end{equation}


\begin{lemma}\label{lemma2.2} For $\alpha>0$, $j>n\ge 0$ and $q> 1$, the following inequality holds:
$$
\sum_{l=n+1}^\infty \left(\frac{k^\alpha(l-n+j)}{k^{\alpha+1}(l)}\right)^q\le C_{\alpha,q} j\left({k^{\alpha}(j)\over k^{\alpha+1}(n)}\right)^q.
$$
\end{lemma}
\begin{proof}  We apply the estimation (\ref{double}) in several places to get
\begin{eqnarray*}
    \sum_{l=n+1}^\infty \left(\frac{k^\alpha(l-n+j)}{k^{\alpha+1}(l)}\right)^q &=& \sum_{m=1}^\infty \left(\frac{k^\alpha(m+j)}{k^{\alpha+1}(m+n)}\right)^q
    \le C_{\alpha,q} \sum_{m=1}^\infty \left(\frac{m+j}{m+n}\right)^{\alpha q} \frac{1}{(m+j)^q}\\
    &<& C_{\alpha,q}\left(\frac{j+1}{n+1}\right)^{\alpha q}  \sum_{m=1}^\infty \frac{1}{(m+j)^q}
    \le C_{\alpha,q} \left(\frac{j+1}{n+1}\right)^{\alpha q} \int_{0}^\infty  \frac{dt}{(t+j)^q}\\
    &=& C_{\alpha,q} \left(\frac{j+1}{n+1}\right)^{\alpha q}\frac{1}{j^{q-1}}\le C_{\alpha,q} j\left({k^{\alpha}(j)\over k^{\alpha+1}(n)}\right)^q,
\end{eqnarray*}
and we conclude the result.
\end{proof}

\begin{theorem} \label{llave} Take $n, u \in \N\cup\{0\}$, $\Real t>0$ or $t=0,$  and $\alpha\in \R_+\backslash\{0,1,2,\ldots\}.$ Then the
the following equality holds
$$
e^{-tu}\sum_{j=\max\{u,n\}}^\infty{j\choose u} k^{-\alpha}(j-n)(1-e^{-t})^{j-u}=e^{-t\alpha}\sum_{j=0}^{\min\{u,n\}}{n\choose j} k^{-\alpha}(u-j)e^{-tj}(1-e^{-t})^{n-j}.
$$

\end{theorem}

\begin{proof} The case $t=0$ is clear. Let $\Real t>0,$ by analytic prolongation, it is enough to show  the formula for $t>0$. First we consider the case that $0\le u\le n$. We apply Laplace transform  and formula (\ref{beta}) in both parts to transform the formula into the equivalent expression
$$
\sum_{j=n}^\infty{j\choose u} k^{-\alpha}(j-n){\Gamma(j-u+1)\Gamma(v+u)\over \Gamma(v+j+1)}=\sum_{j=0}^u{n\choose j} k^{-\alpha}(u-j){\Gamma(j+\alpha+v)\Gamma(n-j+1)\over \Gamma(n+\alpha+v+1)}.
$$
Note that
\begin{eqnarray*}
&\,&\sum_{j=n}^\infty{j\choose u} k^{-\alpha}(j-n){\Gamma(j-u+1)\Gamma(v+u)\over \Gamma(v+j+1)}= \sum_{j=n}^\infty{\Gamma(j-n-\alpha)\over \Gamma(-\alpha)\Gamma(j-n+1)}{j!\over u!}{\Gamma(v+u)\over \Gamma(v+j+1)}\cr
&\,&= {\Gamma(v+u)\over \Gamma(-\alpha)\Gamma(u+1)}\sum_{l=0}^\infty{\Gamma(l-\alpha)\over \Gamma(l+1)}{\Gamma(l+n+1)\over \Gamma(v+l+n+1)}= {\Gamma(v+u)\Gamma(n+1)\Gamma(v+\alpha)\over \Gamma(u+1) \Gamma(v)\Gamma(n+\alpha+v+1)}
\end{eqnarray*}
where we have applied the Lemma \ref{fits}. In the other hand, we have that
\begin{eqnarray*}
&\,&\sum_{j=0}^u{n\choose j} k^{-\alpha}(u-j){\Gamma(j+\alpha+v)\Gamma(n-j+1)\over \Gamma(n+\alpha+v+1)}= \sum_{j=0}^u {n!\over j!} k^{-\alpha}(u-j){\Gamma(j+\alpha+v)\over \Gamma(n+\alpha+v+1)}\cr
&\,&\qquad= {\Gamma(n+1)\Gamma(\alpha+v)\over\Gamma(n+\alpha+v+1)} \sum_{j=0}^u  k^{-\alpha}(u-j)k^{\alpha+v}(j)= {\Gamma(n+1)\Gamma(\alpha+v)\over\Gamma(n+\alpha+v+1)}k^{v}(u)= \cr
&\,&\qquad={\Gamma(n+1)\Gamma(\alpha+v)\Gamma(u+v)\over\Gamma(n+\alpha+v+1)\Gamma(v)\Gamma(u+1)},
\end{eqnarray*}
where we have applied the semigroup property, $k^{\gamma}*k^{\beta}=k^{\gamma+\beta}$ for $\gamma, \beta \in \C$.

Now, we consider the case that  $0\le n< u$. Again, we apply Laplace transform  and formula (\ref{beta})  to obtain  the expression
$$
\sum_{j=u}^\infty{j\choose u} k^{-\alpha}(j-n){\Gamma(j-u+1)\Gamma(v+u)\over \Gamma(v+j+1)}=\sum_{j=0}^n{n\choose j} k^{-\alpha}(u-j){\Gamma(j+\alpha+v)\Gamma(n-j+1)\over \Gamma(n+\alpha+v+1)},
$$
which is equivalent to show that
$$
{\Gamma(v+u)\over \Gamma(u+1)}\sum_{j=u}^\infty {\Gamma(j-n-\alpha)\Gamma(j+1)\over \Gamma(j-n+1)\Gamma(j+v+1)}={\Gamma(n+1)\over \Gamma(v+\alpha+n+1)}\sum_{j=0}^n {\Gamma(u-j-\alpha)\Gamma(v+\alpha+j)\over \Gamma(u-j+1)\Gamma(j+1)}.
$$
Now we claim that
$$
{\Gamma(v+u)\over \Gamma(u+1)}\sum_{j=n}^{u-1} {\Gamma(j-n-\alpha)\Gamma(j+1)\over \Gamma(j-n+1)\Gamma(j+v+1)}={\Gamma(n+1)\over \Gamma(v+\alpha+n+1)}\sum_{j=n+1}^u {\Gamma(u-j-\alpha)\Gamma(v+\alpha+j)\over \Gamma(u-j+1)\Gamma(j+1)},
$$
which is equivalent to
$$
{\Gamma(v+u)\over \Gamma(u+1)}\sum_{l=0}^{u-n-1} {\Gamma(l-\alpha)\Gamma(n+l+1)\over \Gamma(l+1)\Gamma(l+n+v+1)}={\Gamma(n+1)\over \Gamma(v+\alpha+n+1)}\sum_{l=0}^{u-n-1} {\Gamma(l-\alpha)\Gamma(v+\alpha+u-l)\over \Gamma(u-l+1)\Gamma(l+1)}
$$
and we apply Theorem \ref{funda} to conclude the equality.

Then, we have to prove
$$
{\Gamma(v+u)\over \Gamma(u+1)}\sum_{j=n}^\infty {\Gamma(j-n-\alpha)\Gamma(j+1)\over \Gamma(j-n+1)\Gamma(j+v+1)}={\Gamma(n+1)\over \Gamma(v+\alpha+n+1)}\sum_{j=0}^u {\Gamma(u-j-\alpha)\Gamma(v+\alpha+j)\over \Gamma(u-j+1)\Gamma(j+1)},
$$
which is equivalent to
$$
{\Gamma(v+u)\over \Gamma(u+1)}\sum_{l=0}^\infty {\Gamma(l-\alpha)\Gamma(l+n+1)\over \Gamma(l+1)\Gamma(l+n+v+1)}={\Gamma(n+1)\over \Gamma(v+\alpha+n+1)}\sum_{j=0}^u {\Gamma(u-j-\alpha)\Gamma(v+\alpha+j)\over \Gamma(u-j+1)\Gamma(j+1)},
$$
and this identity was proved in the first considered case.
\end{proof}

\section{Fractional finite differences} \label{Weyl}

The usual finite difference of a sequence, $(f(n))_{n\in\N_0},$ is defined by $\Delta f(n):=f(n+1)-f(n)$ for $n\in \N_0$; we iterate,  $\Delta^{m+1}= \Delta^m\Delta$,  to get $$
 \Delta^{m}f(n)=\displaystyle\sum_{j=0}^{m}(-1)^{m-j}\binom{m}{j}f(n+j), \qquad n\in \N_0,$$
for $m\in \N$.  Now we consider $ W=-\Delta$, i.e., $Wf(n)=W^1f(n)=f(n)-f(n+1),$ for $n\in \N_0$ and $W^mf(n)=(-1)^m\Delta^m$
for $m\in \N$. In \cite[Definition 2.2]{ALMV} the authors have extended this concept of finite difference of a sequence to the fractional (non integer) case. We recall it:

\begin{definition}\label{WeylDifference} Let $f: \mathbb{N}_0 \to \C$  and $\alpha>0$ be given. The  Weyl sum of order $\alpha$ of $f$, $W^{-\alpha}f$, is defined by $$W^{-\alpha}f(n):=\displaystyle\sum_{j=n}^{\infty} k^{\alpha}(j-n)f(j),\qquad n\in\N_0,$$
whenever the right hand side makes sense. The  Weyl difference of order $\alpha$ of $f$, $W^\alpha f$, is defined by $$W^{\alpha}f(n)=W^m W^{-(m-\alpha)}f(n)=(-1)^m\Delta^mW^{-(m-\alpha)}f(n),\qquad n\in\N_0,$$ for $m=[\alpha]+1,$  whenever the right hand side makes sense. In particular $W^{\alpha}: c_{0,0}\to c_{0,0}$ for $\alpha \in \R$.
\end{definition}

The following theorem shows an equivalent definition of the fractional Weyl differences.

\begin{theorem} \label{quivalent}Let $f: \mathbb{N}_0 \to \C$ and $\alpha>0$ be given. Then $$W^{\alpha}f(n)=\displaystyle\sum_{j=n}^{\infty} k^{-\alpha}(j-n)f(j),\qquad n\in\N_0,$$ whenever both expressions  make sense.
\end{theorem}
\begin{proof}
If $\alpha\in\N,$ then $$W^{\alpha}f(n)=\displaystyle\sum_{j=0}^{\alpha}(-1)^j\binom{\alpha}{j}f(n+j)=\displaystyle\sum_{j=0}^{\infty}k^{-\alpha}(j)f(n+j)=\displaystyle\sum_{j=n}^{\infty}k^{-\alpha}(n-j)f(j).$$ Now let $m-1<\alpha<m$ with $m\in\N.$ Then \begin{eqnarray*}
&\,&W^{\alpha}f(n)=W^mW^{-(m-\alpha)}f(n)=\displaystyle\sum_{j=0}^{m}(-1)^j\binom{m}{j}\displaystyle\sum_{l=n+j}^{\infty}k^{m-\alpha}(l-n-j)f(l)\\
&=&\displaystyle\sum_{l=n}^{n+m}f(l)\displaystyle\sum_{j=0}^{l-n}(-1)^j\binom{m}{j}k^{m-\alpha}(l-n-j)+\displaystyle\sum_{l=n+m+1}^{\infty}f(l)\displaystyle\sum_{j=0}^{m}(-1)^j\binom{m}{j}k^{m-\alpha}(l-n-j)\\
&=&\displaystyle\sum_{l=n}^{n+m}f(l)(k^{-m}*k^{m-\alpha})(l-n)+\displaystyle\sum_{l=n+m+1}^{\infty}f(l)(k^{-m}*k^{m-\alpha})(l-n)\cr
&=&\displaystyle\sum_{l=n}^{\infty}f(l)k^{-\alpha}(l-n),
\end{eqnarray*}
and we conclude the result.
\end{proof}

The fractional difference $\displaystyle\sum_{j=n}^{\infty} k^{-\alpha}(j-n)f(j)$ has been studied recently in $\ell^p$ and in discrete H\"older spaces, see \cite{ADT}. For example, observe that for $0<\alpha<1$ and the sequence $k^1,$ the series $\displaystyle\sum_{j=n}^{\infty} k^{-\alpha}(j-n)k^1(j)$ is convergent, meanwhile $W^{\alpha}k^1$ is not defined.

Now, we present a  technical (and interesting) result which we will need in the proofs of the main theorems..

\begin{proposition}\label{leibniz} If $\alpha \in \R$ and $f\in c_{0,0},$ then $$W^{\alpha}(jf(j))(n)=(n+\alpha)W^{\alpha}f(n)-\alpha W^{\alpha-1}f(n),\quad n\in\N_0.$$

\end{proposition}
\begin{proof} Let $\alpha>0,$ then \begin{eqnarray*}
W^{-\alpha}(jf(j))(n)&=&\displaystyle\sum_{j=n}^{\infty}k^{\alpha}(j-n)(j+\alpha-n)f(j)+\displaystyle\sum_{j=n}^{\infty}k^{\alpha}(j-n)(n-\alpha)f(j) \\
&=& \alpha W^{-\alpha-1}f(n)+(n-\alpha)W^{-\alpha}f(n),
\end{eqnarray*}
for $n\in \N_0,$ and the proposition is proved for negative  $\alpha$. Note that $$W^{-\alpha}(jW^{\alpha}f(j)-\alpha W^{\alpha-1}f(j))=\alpha W^{-1}f(n)+(n-\alpha)f(n)-\alpha W^{-1}f(n)=(n-\alpha)f(n),$$ we  apply $W^{\alpha}$ to the above identity, and we get the result.
\end{proof}

\begin{example}\label{ex22}

\begin{itemize} \item[(i)] {\rm Let $\lambda \in \mathbb{C}\backslash\{0\},$ and $r_{\lambda}(n):=\lambda^{-(n+1)}$ for $n\in\N_0.$ Note that $r_{\lambda}$ are eigenfunctions for the operator $W^{\alpha}$ for $\alpha\in \R$ and $\vert \lambda\vert >1$, i.e.,
$$
W^{\alpha}r_{\lambda}= \frac{(\lambda-1)^{\alpha}}{\lambda^{\alpha}} r_{\lambda},\qquad |\lambda|>1,
$$
see details \cite[Example 2.5(i)]{ALMV}}

\item[(ii)]{\rm Let $\alpha\in \R$ and $m\in\N,$ then $$W^m k^{\alpha}(n)=(-1)^{m}k^{\alpha-m}(n+m),\quad n\in\N_0.$$
Note that $k^{\alpha}(0)-k^{\alpha}(1)=1-\alpha=-k^{\alpha-1}(1)$. Let $n\in\N,$ we have $$
W k^{\alpha}(n)=k^{\alpha}(n)-k^{\alpha}(n+1)=\frac{\alpha(\alpha+1)\cdots(\alpha+n-1)}{n!}\biggl( 1-\frac{\alpha+n}{n+1}\biggr)=-k^{\alpha-1}(n+1).$$ Then we iterate to show that $W^m k^{\alpha}(n)=(-1)^{m}k^{\alpha-m}(n+m)$ for $m, n\in\N_0$.
}
\end{itemize}
\end{example}

\section{Semigroups on sequence spaces and fractional finite differences}

In this section we study the fractional Weyl differences of the one-parameter operator families $(T(t))_{t\geq 0}$ and $(S(t))_{t\geq 0}$ given by $$T(t)f(n):=\displaystyle\sum_{j=0}^n\binom{n}{j}e^{-tj}(1-e^{-t})^{n-j}f(j),$$ $$S(t)f(n):=e^{-tn}\displaystyle\sum_{j=n}^{\infty}\binom{j}{n}(1-e^{-t})^{j-n}f(j),$$ for $ n\in\N_0,$ $t\geq 0,$ and $(f(n))_{n\in\N_0}$ a sequence where the above operator families are defined. They will play a key role in Section \ref{sect5}. It is a simple check that $(T(t))_{t\geq 0}$ and $(S(t))_{t\geq 0}$ have the semigroup property.

Let $f:\N_0\to \C$ be a scalar sequence. {We  recall that} the $Z$-transform of $f$ is defined by
\begin{equation} \label{zeta}
\tilde f(z)= \sum_{n=0}^{\infty} f(n) z^{n},
\end{equation}
for all $z$ such that this series converges. The set of numbers $z$ in the complex plane for which {the} series \eqref{zeta} converges
is called the region of convergence of $\tilde f$.

By (\ref{generating}),  note that $\widetilde{ k^{\alpha}}(z)= \displaystyle{1 \over (1-z)^{\alpha}},$ for $\vert z\vert <1.$  We check the $Z$-transform of semigroups $(T(t))_{t\geq 0}$ and $(S(t))_{t\geq 0}$.

\begin{theorem} Take $f:\N_0\to \C$ such $\widetilde{f}(z)$ exists for $\vert z\vert <1$. Then
\begin{eqnarray*}
\widetilde{ T(t)f}(z)&=& \widetilde{f}\left({e^{-t}z\over 1-z(1-e^{-t})}\right), \cr
\widetilde{ S(t)f}(z)&=& \widetilde{f}\left(e^{-t}(z-1)+1\right),
\end{eqnarray*}
for $\vert z\vert <1$ and $t>0$.
\end{theorem}

\begin{proof}  By Fubini theorem, we have that
\begin{eqnarray*}
&\,&\widetilde{ T(t)f}(z)=\sum_{n=0}^\infty \left(\sum_{j=0}^n\binom{n}{j}e^{-tj}(1-e^{-t})^{n-j}f(j)\right)z^n\cr
&\,&=\sum_{j=0}^\infty e^{-tj} f(j)z^j \left(\sum_{l=0}^\infty\binom{j+l}{j}(1-e^{-t})^{l}z^l\right)=\sum_{j=0}^\infty e^{-tj} f(j)z^j \left(\sum_{l=0}^\infty k^{j+1}(l)((1-e^{-t})z)^{l}\right)\cr
&\,&={1\over 1-z(1-e^{-t})}\sum_{j=0}^\infty  f(j)\left({e^{-t}z\over 1-z(1-e^{-t})}\right)^j= {1\over 1-z(1-e^{-t})}\widetilde{f}\left({e^{-t}z\over 1-z(1-e^{-t})}\right),
\end{eqnarray*}
where we have applied the formula (\ref{generating}).

Now, again by Fubini theorem, one gets
\begin{eqnarray*}
&\,&\widetilde{ S(t)f}(z)=\sum_{n=0}^\infty \left(e^{-tn}\displaystyle\sum_{j=n}^{\infty}\binom{j}{n}(1-e^{-t})^{j-n}f(j)\right)z^n\cr
&\,&=\sum_{j=0}^\infty (1-e^{-t})^j f(j) \left(\sum_{n=0}^j\binom{j}{n}\left({e^{-t}z\over 1-e^{-t}}\right)^n\right)\cr
&\,&=\sum_{j=0}^\infty (1-e^{-t})^j f(j) \left(1 +{e^{-t}z\over 1-e^{-t}}\right)^j= \widetilde{f}\left(e^{-t}(z-1)+1\right),
\end{eqnarray*}
and we conclude the proof.
\end{proof}

\begin{remark} \label{compos}{\rm Semigroups  of analytic self maps on the unit  disc $\D$ have been considered in detail in the survey \cite{Sis}. In particular, the image of  the semigroups  $(T(t))_{t\geq 0}$ and $(S(t))_{t\geq 0}$ via the $Z$-transform are $(\psi_t)_{t>0}$ and $(\phi_t)_{t>0},$ given by
$$
\psi_t(z):={{e^{-t}z\over 1-z(1-e^{-t})}}, \qquad \phi_t(z): ={e^{-t}(z-1)+1},
$$
for $t>0$ and $z\in \D$, which have been also introduced in previous papers. In \cite[Section 3]{PLMSSis}, the semigroup $(\psi_t)_{t>0}$ is introduced to study the Ces\`{a}ro operator on the space $H^p$ on the disc $\D$ an on the Bergman space in \cite{ArchivSis}, see also \cite[Theorem 3.2]{Xiao}.  }

\end{remark}

\begin{lemma}\label{WeylSem} Let $\alpha\geq 0$ and $f\in c_{0,0}.$ Then the following equality holds \begin{equation}\label{eq0}W^{\alpha}S(t)f(n)=e^{-tn}\displaystyle\sum_{j=n}^{\infty}\binom{j+\alpha}{n+\alpha}(1-e^{-t})^{j-n}W^{\alpha}f(j)\quad t\geq 0,\,n\in\N_0.\end{equation}
As a consequence, we have that
$$
k^{\alpha+1}W^{\alpha}(S(t)f)=S(t)(k^{\alpha+1}W^{\alpha}f), \qquad t\geq 0.$$
\end{lemma}
\begin{proof}
First we suppose that $\alpha=1.$ Then \begin{eqnarray*}
W S(t)f(n)&=& e^{-tn}\displaystyle\sum_{j=n}^{\infty}\binom{j}{n}(1-e^{-t})^{j-n}f(j)-e^{-t(n+1)}\displaystyle\sum_{j=n+1}^{\infty}\binom{j}{n+1}(1-e^{-t})^{j-n-1}f(j)\\
&=&e^{-tn}\displaystyle\sum_{j=n}^{\infty}\binom{j}{n}(1-e^{-t})^{j-n}f(j)+e^{-tn}\displaystyle\sum_{j=n+1}^{\infty}\binom{j}{n+1}(1-e^{-t})^{j-n}f(j) \\
&&-e^{-tn}\displaystyle\sum_{j=n+1}^{\infty}\binom{j}{n+1}(1-e^{-t})^{j-n-1}f(j),
\end{eqnarray*} where we have used that $e^{-t(n+1)}=e^{-tn}-e^{-tn}(1-e^{-t}).$ Note that $\binom{j}{n}+\binom{j}{n+1}=\binom{j+1}{n+1},$ and then we get \begin{eqnarray*}
W S(t)f(n)&=& e^{-tn}\biggl(f(n)+ \displaystyle\sum_{j=n+1}^{\infty}\binom{j+1}{n+1}(1-e^{-t})^{j-n}f(j)\\
&&-\displaystyle\sum_{j=n+1}^{\infty}\binom{j}{n+1}(1-e^{-t})^{j-n-1}f(j)\biggr)\\
&=&e^{-tn} \displaystyle\sum_{j=n}^{\infty}\binom{j+1}{n+1}(1-e^{-t})^{j-n}(f(j)-f(j+1)),
\end{eqnarray*}
where we have done a change of variable in the last step.

Now we proceed by induction. Let $m\in\N,$ and we suppose \eqref{eq0} true for $\alpha\in\N$ with $\alpha<m.$ Then using that $W^m=WW^{m-1}$ and the same arguments that in the above case we have  \begin{eqnarray*}
W^m S(t)f(n)&=& W^{m-1}S(t)f(n)-W^{m-1}S(t)f(n+1)\\
&=&e^{-tn}\displaystyle\sum_{j=n}^{\infty}\binom{j+m-1}{n+m-1}(1-e^{-t})^{j-n}W^{m-1}f(j)\\
&&+e^{-tn}\displaystyle\sum_{j=n+1}^{\infty}\binom{j+m-1}{n+m}(1-e^{-t})^{j-n}W^{m-1}f(j) \\
&&-e^{-tn}\displaystyle\sum_{j=n+1}^{\infty}\binom{j+m-1}{n+m}(1-e^{-t})^{j-n-1}W^{m-1}f(j)\\
&=&e^{-tn} \displaystyle\sum_{j=n}^{\infty}\binom{j+m}{n+m}(1-e^{-t})^{j-n}(W^{m-1}f(j)-W^{m-1}f(j+1)).
\end{eqnarray*}

Finally we prove the general case. Let $m<\alpha<m+1$ with $m=[\alpha].$ We write \begin{eqnarray*}
W^{\alpha}S(t)f(n)&=&W^{\alpha-m}W^m S(t)f(n)\\
&=&\displaystyle\sum_{j=n}^{\infty}k^{1+m-\alpha}(j-n)W^m S(t)f(j)-\displaystyle\sum_{j=n+1}^{\infty}k^{1+m-\alpha}(j-n-1)W^m S(t)f(j)\\
&=&\displaystyle\sum_{j=n}^{\infty}(-1)^{j-n}\binom{\alpha-m}{j-n}W^m S(t)f(j)\\
&=&\displaystyle\sum_{j=n}^{\infty}(-1)^{j-n}\binom{\alpha-m}{j-n}e^{-tj}\displaystyle\sum_{v=j}^{\infty}\binom{v+m}{j+m}(1-e^{-t})^{v-j}W^mf(v)\\
&=&e^{-tn}\displaystyle\sum_{j=n}^{\infty}\binom{\alpha-m}{j-n}\sum_{u=0}^{j-n}(-1)^{u}\binom{j-n}{u}\displaystyle\sum_{v=j}^{\infty}\binom{v+m}{j+m}(1-e^{-t})^{v-n-u}W^mf(v),
\end{eqnarray*}
where we have applied $W^{\alpha}=W^{\alpha-m}W^m$ (\cite[Remark 2.3, Proposition 2.4]{ALMV}), Example \ref{ex22} (ii) and the equality $e^{-tj}=e^{-tn}\sum_{u=0}^{j-n}(-1)^{j-n-u}\binom{j-n}{u}(1-e^{-t})^{j-n-u}.$ Applying twice Fubini's Theorem and a change of variable we have \begin{displaymath}\begin{array}{l}
\displaystyle W^{\alpha}S(t)f(n)=e^{-tn}\displaystyle\sum_{u=0}^{\infty}(-1)^{u}\displaystyle\sum_{v=n+u}^{\infty}(1-e^{-t})^{v-n-u}W^mf(v)\sum_{j=n+u}^{v}\binom{\alpha-m}{j-n}\binom{j-n}{u}\binom{v+m}{j+m}\\
\displaystyle =e^{-tn}\displaystyle\sum_{u=0}^{\infty}(-1)^{u}\displaystyle\sum_{v=n+u}^{\infty}(1-e^{-t})^{v-n-u}W^mf(v)\sum_{l=0}^{v-n-u}\binom{\alpha-m}{l+u}\binom{l+u}{u}\binom{v+m}{l+u+n+m}.
\end{array}\end{displaymath}
It is a simple check that $\binom{\alpha-m}{l+u}\binom{l+u}{u}=\binom{\alpha-m}{u}\binom{\alpha-m-u}{l}.$ Then using the Chu-Vandermonde's identity, Fubini's Theorem and a change of variable we have \begin{eqnarray*}
W^{\alpha}S(t)f(n)&=&e^{-tn}\displaystyle\sum_{u=0}^{\infty}(-1)^{u}\binom{\alpha-m}{u}\displaystyle\sum_{v=n+u}^{\infty}(1-e^{-t})^{v-n-u}W^mf(v)\binom{v-u+\alpha}{n+\alpha}\\
&=&e^{-tn}\displaystyle\sum_{u=0}^{\infty}(-1)^{u}\binom{\alpha-m}{u}\displaystyle\sum_{j=n}^{\infty}(1-e^{-t})^{j-n}W^mf(j+u)\binom{j+\alpha}{n+\alpha}\\
&=&e^{-tn}\displaystyle\sum_{j=n}^{\infty}(1-e^{-t})^{j-n}\binom{j+\alpha}{n+\alpha}\displaystyle\sum_{u=0}^{\infty}(-1)^{u}\binom{\alpha-m}{u}W^mf(j+u)\\
&=&e^{-tn}\displaystyle\sum_{j=n}^{\infty}(1-e^{-t})^{j-n}\binom{j+\alpha}{n+\alpha}W^{\alpha}f(j),\\
\end{eqnarray*}
and we conclude the equality.

As $\displaystyle{k^{\alpha+1}(n){j+\alpha \choose n+\alpha}=k^{\alpha+1}(j){j \choose n} }$ with $j\ge n$ and $\alpha \ge 0$, we have that
\begin{eqnarray*}
&\,&k^{\alpha+1}(n)W^{\alpha}(S(t)f)(n)=e^{-tn}\displaystyle\sum_{j=n}^{\infty}(1-e^{-t})^{j-n}k^{\alpha+1}(n)\binom{j+\alpha}{n+\alpha}W^{\alpha}f(j)\cr
&\,&\quad=e^{-tn}\displaystyle\sum_{j=n}^{\infty}\binom{j}{n}(1-e^{-t})^{j-n}k^{\alpha+1}(j)W^{\alpha}f(j)
=S(t)(k^{\alpha+1}W^{\alpha}f)(n)
\end{eqnarray*}
for $n\in \N_0$ and the proof is finished.\end{proof}

\begin{lemma}\label{WeylSemT} Let $\alpha\geq 0$ and $f\in c_{0,0}.$ The the following equality holds \begin{equation}\label{eq1}W^{\alpha}T(t)f(n)=e^{-t\alpha}T(t)W^{\alpha}f(n),\quad t\geq 0,\,n\in\N_0.\end{equation}
\end{lemma}
\begin{proof}
First we suppose that $\alpha=1.$ Then \begin{eqnarray*}
W T(t)f(n)&=& \displaystyle\sum_{j=0}^{n}\binom{n}{j}e^{-tj}(1-e^{-t})^{n-j}f(j)-e^{-t}\displaystyle\sum_{j=0}^{n+1}\binom{n+1}{j}e^{-tj}(1-e^{-t})^{n+1-j}f(j)\\
&=&\displaystyle\sum_{j=0}^{n}\binom{n}{j}e^{-tj}(1-e^{-t})^{n-j}f(j)-(1-e^{-t})^{n+1}f(0)\\
&&-\displaystyle\sum_{j=1}^{n}\binom{n}{j}e^{-tj}(1-e^{-t})^{n+1-j}f(j) \\
&&-\displaystyle\sum_{j=1}^{n}\binom{n}{j-1}e^{-tj}(1-e^{-t})^{n+1-j}f(j)-e^{-t(n+1)}f(n+1),
\end{eqnarray*} where we have used that $\binom{n+1}{j}=\binom{n}{j}+\binom{n}{j-1}.$ Note that $(1-e^{-t})^{n+1-j}=(1-e^{-t})^{n-j}-e^{-t}(1-e^{-t})^{n-j},$ and then we get \begin{eqnarray*}
W T(t)f(n)&=& \displaystyle\sum_{j=0}^{n}\binom{n}{j}e^{-tj}(1-e^{-t})^{n-j}f(j)-(1-e^{-t})^{n+1}f(0)\\
&&-\displaystyle\sum_{j=1}^{n}\binom{n}{j}e^{-tj}(1-e^{-t})^{n-j}f(j)+\displaystyle\sum_{j=1}^{n}\binom{n}{j}e^{-t(j+1)}(1-e^{-t})^{n-j}f(j)\\
&&-\displaystyle\sum_{j=1}^{n}\binom{n}{j-1}e^{-tj}(1-e^{-t})^{n+1-j}f(j)-e^{-t(n+1)}f(n+1)\\
&=&(1-e^{-t})^{n}f(0)-(1-e^{-t})^{n+1}f(0)+\displaystyle\sum_{j=1}^{n}\binom{n}{j}e^{-t(j+1)}(1-e^{-t})^{n-j}f(j)\\
&&-\displaystyle\sum_{j=0}^{n-1}\binom{n}{j}e^{-t(j+1)}(1-e^{-t})^{n-j}f(j+1)-e^{-t(n+1)}f(n+1)\\
&=&\displaystyle\sum_{j=0}^{n}\binom{n}{j}e^{-t(j+1)}(1-e^{-t})^{n-j}W f(j)=e^{-t}T(t)W f(n),
\end{eqnarray*}
where we have done a change of variable.

Now we proceed by induction. Let $m\in\N,$ and we suppose \eqref{eq1} true for $\alpha\in\N$ with $\alpha<m.$ Then one gets  \begin{eqnarray*}
W^m T(t)f(n)&=& \displaystyle\sum_{j=0}^{n}\binom{n}{j}e^{-t(j+m-1)}(1-e^{-t})^{n-j}W^{m-1}f(j)\\
&&-\displaystyle\sum_{j=0}^{n+1}\binom{n+1}{j}e^{-t(j+m-1)}(1-e^{-t})^{n+1-j}W^{m-1}f(j)\\
&=&\displaystyle\sum_{j=0}^{n}\binom{n}{j}e^{-t(j+m-1)}(1-e^{-t})^{n-j}W^{m-1}f(j)-e^{-t(m-1)}(1-e^{-t})^{n+1}W^{m-1}f(0)\\
&&-\displaystyle\sum_{j=1}^{n}\binom{n}{j}e^{-t(j+m-1)}(1-e^{-t})^{n+1-j}W^{m-1}f(j)\\
&&-\displaystyle\sum_{j=1}^{n}\binom{n}{j-1}e^{-t(j+m-1)}(1-e^{-t})^{n+1-j}W^{m-1}f(j)-e^{-t(n+m)}W^{m-1}f(n+1)\\
&=&\displaystyle\sum_{j=0}^{n}\binom{n}{j}e^{-t(j+m-1)}(1-e^{-t})^{n-j}W^{m-1}f(j)-e^{-t(m-1)}(1-e^{-t})^{n+1}W^{m-1}f(0)\\
&&-\displaystyle\sum_{j=1}^{n}\binom{n}{j}e^{-t(j+m-1)}(1-e^{-t})^{n-j}W^{m-1}f(j)\\
&&+\displaystyle\sum_{j=1}^{n}\binom{n}{j}e^{-t(j+m)}(1-e^{-t})^{n-j}W^{m-1}f(j)\\
&&-\displaystyle\sum_{j=1}^{n}\binom{n}{j-1}e^{-t(j+m-1)}(1-e^{-t})^{n+1-j}W^{m-1}f(j)-e^{-t(n+m)}W^{m-1}f(n+1)\\
&=&e^{-t(m-1)}(1-e^{-t})^{n}W^{m-1}f(0)-e^{-t(m-1)}(1-e^{-t})^{n+1}W^{m-1}f(0)\\
&&+\displaystyle\sum_{j=1}^{n}\binom{n}{j}e^{-t(j+m)}(1-e^{-t})^{n-j}W^{m-1}f(j)\\
&&-\displaystyle\sum_{j=0}^{n-1}\binom{n}{j}e^{-t(j+m)}(1-e^{-t})^{n-j}W^{m-1}f(j+1)-e^{-t(n+m)}W^{m-1}f(n+1)\\
&=&e^{-tm}\displaystyle\sum_{j=0}^{n}\binom{n}{j}e^{-tj}(1-e^{-t})^{n-j}W^{m}f(j)=e^{-tm}T(t)W^{m}f(n).
\end{eqnarray*}

Finally we prove the case that $\alpha \in \R^{+}\backslash (\N\cup\{0\})$. By Theorem \ref{quivalent} and Theorem \ref{llave} we write \begin{eqnarray*}
&\,&W^{\alpha}T(t)f(n)=\sum_{j=n}^{\infty}k^{-\alpha}(j-n)T(t)f(j)=\sum_{j=n}^{\infty}k^{-\alpha}(j-n)\sum_{u=0}^j{j\choose u}e^{-tu}(1-e^{-t})^{j-u}f(u)\\
&\,&=\sum_{u=0}^{\infty}f(u)e^{-tu}\sum_{j=\max\{u,n\}}^\infty {j\choose u}k^{-\alpha}(j-n)(1-e^{-t})^{j-u}\\
&\,&=\sum_{u=0}^{\infty}f(u)e^{-t\alpha}\sum_{j=0}^{\min\{u,n\}}{n\choose j} k^{-\alpha}(u-j)e^{-tj}(1-e^{-t})^{n-j}\\
&\,&=\sum_{j=0}^{n}{n\choose j}e^{-t(\alpha+j)}(1-e^{-t})^{n-j}\sum_{u=j}^{\infty} k^{-\alpha}(u-j)f(u)\\
&\,&=\sum_{j=0}^{n}{n\choose j}e^{-t(\alpha+j)}(1-e^{-t})^{n-j}W^\alpha f(j)=e^{-t\alpha}T(t)W^{\alpha}f(n),
\end{eqnarray*}
and we conclude the result.
\end{proof}

\section{Sobolev-Lebesgue sequence spaces}\label{sect2}

In this section we introduce a family of subspaces $\tau_p^\alpha$ which are contained in $\ell^p$ for $\alpha \ge 0$ and $1\le p\le \infty$. Note that this definition includes the usual Lebesgue sequence space $\ell^p$ as limit case for  $\alpha=0$. For $p=1$ these spaces have been considered in \cite[Theorem 2.11]{ALMV} and \cite[Section 2]{A}.

\begin{definition}\label{taualphasp}
For $\alpha>0$ and $1\leq p<\infty$ let $\tau_p^{\alpha}$ be the Banach space formed by the set of complex sequences vanishing at infinity such that the norm
$$\|f\|_{\alpha,p}:=\left(\sum_{n=0}^\infty |W^\alpha f(n)|^p(k^{\alpha+1}(n))^p\right)^{1/p}$$ converges. For $p=\infty,$ we denote by $\tau_{\infty}^{\alpha}$ the Banach space obtained as the complex functions vanishing at infinity such that the norm
$$\|f\|_{\alpha,\infty}:=\displaystyle\sup_{n\in\N_0}|W^\alpha f(n)|k^{\alpha+1}(n)$$ converges.
\end{definition}

\begin{theorem}\label{propiedades} Let $\alpha>0$ and $1\leq p\le \infty.$ Then \begin{enumerate}

\item[(i)] The operator $D^{\alpha}:\tau_p^{\alpha}\to\ell^p$ defined by $$D^{\alpha}f(n):=k^{\alpha+1}(n)W^{\alpha}f(n),\quad n\in\N_0,\, f\in \tau_p^{\alpha},$$ is an isometry whose inverse operator $(D^{\alpha})^{-1}: \ell^p\to  \tau_p^{\alpha}$ is given by
    $$(D^{\alpha})^{-1}f(n)=W^{-\alpha}((k^{\alpha+1})^{-1}f)(n),\quad n\in\N_0,\, f\in \ell^p.$$

 \item[(ii)]The following embeddings hold: $\tau_1^{\alpha}\hookrightarrow\tau_p^{\alpha}\hookrightarrow\tau_q^{\alpha}\hookrightarrow\tau_{\infty}^{\alpha},$ for $1<p<q<\infty.$

\item[(iii)] The following embeddings hold: $\tau_p^{\beta}\hookrightarrow\tau_p^{\alpha}\hookrightarrow\ell^p,$ for $\beta>\alpha>0$.

\item[(iv)] For $p\ge 1$ and $\frac{1}{p}+\frac{1}{p'}=1,$ (with $p'=\infty$ for $p=1$) the space $\tau_{p'}^{\alpha}$ is the dual space of $\tau_p^{\alpha},$ and  the duality is given by $$\langle f,g \rangle_{\alpha}:= \displaystyle\sum_{n=0}^{\infty} W^{\alpha}f(n)W^{\alpha}g(n)\left(k^{\alpha+1}(n)\right)^2, \quad f\in \tau_p^{\alpha},\, g\in \tau_{p'}^{\alpha}.$$
\end{enumerate}
\end{theorem}
\begin{proof} (i) By definition, we have
\begin{equation*}
    \|D^{\alpha}f\|_p=\left(\sum_{n=0}^\infty (k^{\alpha+1}(n))^p|W^{\alpha}f(n)|^p\right)^{1/p}=\|f\|_{\alpha,p}.
\end{equation*}

Now, we use part (i) to show (ii). As $\ell^1\hookrightarrow \ell^p\hookrightarrow\ell^q\hookrightarrow\ell^\infty $, for $1<p<q<\infty,$ we have that
$$
\Vert f\Vert_{\alpha, \infty}=\Vert D^\alpha f\Vert_{\infty}\le \Vert D^\alpha f\Vert_{q}=\Vert  f\Vert_{\alpha,q}\le\Vert D^\alpha f\Vert_{p}=\Vert  f\Vert_{\alpha,p}\le \Vert D^\alpha f\Vert_{1}=\Vert f\Vert_{\alpha,1}.
$$

\noindent (iii) Let $f\in c_{0,0}$ and $0\le\alpha<\beta.$ For $p=1$ see \cite[Theorem 2.11]{ALMV}. For $p=\infty,$ \begin{eqnarray*}
\|f\|_{\alpha,\infty}&\leq& \sup_{n\in\N_0}\sum_{j=n}^\infty k^{\beta-\alpha}(j-n)|W^\beta f(n)|k^{\alpha+1}(n) \\
&\leq& \sup_{n\in\N_0}|W^\beta f(n)|\sum_{j=n}^\infty k^{\beta-\alpha}(j-n)k^{\alpha+1}(j)=\|f\|_{\alpha,\infty},
\end{eqnarray*} since $k^{\alpha+1}$ is increasing. Let $1<p<\infty,$ then
\begin{displaymath}
\begin{array}{l}
\displaystyle\|f\|_{\alpha,p}\le\left(|W^{\alpha}f(0)|^p+\sum_{n=1}^\infty \left(\sum_{j=n}^\infty k^{\beta-\alpha}(j-n)|W^\beta f(j)|k^{\alpha+1}(n)\right)^p\right)^{1/p}\\
\displaystyle\le C_{\beta, \alpha}\left(\left(\sum_{j=0}^{\infty}\frac{k^{\beta-\alpha}(j)}{k^{\beta+1}(j)}k^{\beta+1}(j)|W^{\beta}f(j)|\right)^p+\sum_{n=1}^\infty \left(\sum_{j=n}^\infty \frac{(j-n)^{\beta-\alpha-1}}{\Gamma(\beta-\alpha)}|W^\beta f(j)|k^{\alpha+1}(j)\right)^p\right)^{1/p}\\
\displaystyle\le C_{\beta, \alpha,p}\left(\|f\|_{\beta,p}^p +\sum_{n=1}^\infty \left(n^{\beta-\alpha}|W^\beta f(n)|k^{\alpha+1}(n)\right)^p\right)^{1/p}\\
\displaystyle\le C_{\beta, \alpha,p}\|f\|_{\beta,p},
\end{array}\end{displaymath}
where we have used \eqref{double}, H\"older's inequality, Hardy's inequality \eqref{hardyDual} and $k^{\alpha+1}$ is increasing.

\noindent (iv) Combine  that  the operator $D^\alpha$ is an isometry and $(\ell^p)'\cong \ell^{p'}$ for $ p\ge 1$.
\end{proof}

\begin{lemma}\label{belong-p} Let $\alpha\geq 0.$

\begin{itemize}

\item[(i)] For $1< p<\infty,$ the function $k^{\beta}\in \tau_p^{\alpha}$ if and only if $\Real\beta<1-\frac{1}{p}.$ Furthermore $k^{\beta}\in \tau_1^{\alpha}$ if and only if $\Real\beta<0$ or $\beta=0,$ and $k^{\beta}\in \tau_{\infty}^{\alpha}$ if and only if $\Real\beta\leq 1.$

    \item[(ii)] For $1\leq p<\infty$ and $\vert \lambda\vert >1$, the sequence $r_{\lambda}\in \tau_p^{\alpha}$ and
        $$
        \Vert r_\lambda\Vert_{\alpha, p}\le C_{\alpha, p}\left({\vert \lambda^p-\lambda^{p-1}\vert\over \vert \lambda\vert^p-1}\right)^{\alpha}{1\over (\vert \lambda\vert^p-1)^{1\over p}}.
        $$
        Finally $r_{\lambda}\in \tau_\infty^{\alpha}$ for $\vert \lambda\vert >1$
        and
        $$
         \Vert r_\lambda\Vert_{\alpha, \infty} \le\alpha^\alpha{\vert \lambda-1\vert^\alpha\over \vert \lambda\vert^{2\alpha+1}}.
        $$
\end{itemize}
\end{lemma}
\begin{proof}
(i) Let $1<p<\infty.$ First note that for $\Real\beta\geq 1,$ $k^{\beta}\notin\ell^p$ by the estimation \eqref{double2}, and then $k^{\beta}\notin\tau_p^{\alpha}.$ Secondly observe that for $\beta\in\{0,-1,-2,\ldots\},$ $k^{\beta}\in c_{0,0},$ then $k^{\beta}\in\tau_p^{\alpha}.$ Finally for $\Real\beta< 1$ with $\beta\notin\{0,-1,-2,\ldots\},$ using Proposition \ref{propiedades} (iii) it is enough to prove the result for $\alpha\in\N_0.$ Let $\alpha\in\N_0,$ then $$\displaystyle\sum_{n=0}^{\infty}|W^{\alpha}k^{\beta}(n)|^p(k^{\alpha+1}(n))^p=\displaystyle\sum_{n=0}^{\infty}|k^{\beta-\alpha}(n+\alpha)|^p(k^{\alpha+1}(n))^p<\infty$$ if and only if $\Real\beta<1-\frac{1}{p},$ where we have applied Example \ref{ex22} (ii) and \eqref{double2}. The cases $p=1$ and $p=\infty$ use the same arguments than the previous one.
(ii) It is similar to the case (i) using Example \ref{ex22} (i).
\end{proof}

\begin{theorem}\label{propiedades2} Let $\alpha>0$ and $1\leq p\le\infty.$ Then
$$
\Vert f\ast g\Vert_{\alpha, p}\le C_{\alpha,p}\Vert f\Vert_{ \alpha,p}\Vert g\Vert_{ \alpha,1}, \quad f\in \tau_p^{\alpha}, \ g \in \tau_1^{\alpha}.
$$
\end{theorem}
\begin{proof}
Let $f\in\tau_p^{\alpha}$ and $g\in\tau_1^{\alpha}$. The case $p=1$ was proved in \cite[Theorem 2.10]{ALMV}. Let us consider $p>1$ and $q>1$ such that $\frac{1}{p}+\frac{1}{q}=1$. By \cite[Lemma 2.7]{ALMV} (see also \cite[Lemma 4.4]{Gale}),
\begin{equation*}
    W^\alpha(f*g)(n)=\sum_{j=0}^n W^\alpha g(j)\sum_{u=n-j}^n k^\alpha(u-n+j)W^\alpha f(u) - \sum_{j=n+1}^\infty W^\alpha g(j)\sum_{u=n+1}^\infty k^\alpha(u-n+j)W^\alpha f(u).
\end{equation*}
Then, by Minkwoski's inequality (\cite[Theorem 25, p.31]{HLP}),
\begin{eqnarray}\label{ineq}
    \|f*g\|_{\alpha,p}&=& \left(\sum_{n=0}^\infty |W^\alpha(f*g)(n)|^p (k^{\alpha+1}(n))^p\right)^{1/p}\nonumber\\
    &\le& \left(\sum_{n=0}^\infty \left(\sum_{j=0}^n |W^\alpha g(j)|\sum_{u=n-j}^n k^\alpha(u-n+j)|W^\alpha f(u)|\right)^p (k^{\alpha+1}(n))^p\right)^{1/p} \nonumber\\
    &&+ \left(\sum_{n=0}^\infty \left(\sum_{j=n+1}^\infty |W^\alpha g(j)|\sum_{u=n+1}^\infty k^\alpha(u-n+j)|W^\alpha f(u)|\right)^p (k^{\alpha+1}(n))^p\right)^{1/p}.
\end{eqnarray}
Now, we apply a inequality of Minkwoski type (\cite[Theorem 165, p.123]{HLP}) to the second addend of (\ref{ineq}) and we obtain
\begin{eqnarray*}
    &&\left(\sum_{n=0}^\infty (k^{\alpha+1}(n))^p\left(\sum_{j=n+1}^\infty |W^\alpha g(j)|\sum_{u=n+1}^\infty k^\alpha(u-n+j)|W^\alpha f(u)|\right)^p \right)^{1/p}\\
    &&\le \sum_{j=1}^\infty |W^\alpha g(j)|\left(\sum_{n=0}^{j-1}(k^{\alpha+1}(n))^p \left(\sum_{u=n+1}^\infty k^\alpha(u-n+j)|W^\alpha f(u)|\right)^p \right)^{1/p}.
\end{eqnarray*}
By using H\"older's inequality (\cite[Theorem 13, p.24]{HLP}),
\begin{equation*}
    \sum_{u=n+1}^\infty k^\alpha(u-n+j)|W^\alpha f(u)|\le \left(\sum_{u=n+1}^\infty \left(\frac{k^\alpha(u-n+j)}{k^{\alpha+1}(u)}\right)^q\right)^{1/q}
    \left(\sum_{u=n+1}^\infty \left(k^{\alpha+1}(u)|W^\alpha f(u)|\right)^p\right)^{1/p},
\end{equation*}
where
$$\left(\sum_{u=n+1}^\infty \left(k^{\alpha+1}(u)|W^\alpha f(u)|\right)^p\right)^{1/p}\le \|f\|_{\alpha,p}.$$
By Lemma \ref{lemma2.2} with $0\leq n<j$ and \eqref{double} one gets
\begin{equation*}
    \sum_{u=n+1}^\infty k^\alpha(u-n+j)|W^\alpha f(u)|\leq C_{\alpha,p}\|f\|_{\alpha,p}\frac{k^{\alpha+1}(j)}{k^{\alpha+1}(n)}\frac{1}{j^{1/p}},
\end{equation*}
and then
\begin{eqnarray*}
    &&\left(\sum_{n=0}^{j-1}(k^{\alpha+1}(n))^p \left(\sum_{u=n+1}^\infty k^\alpha(u-n+j)|W^\alpha f(u)|\right)^p \right)^{1/p}\\
    &&\leq C_{\alpha,p}\|f\|_{\alpha,p}k^{\alpha+1}(j)\left(\sum_{n=0}^{j-1}\frac{1}{j}\right)^{1/p}=C_{\alpha,p}\|f\|_{\alpha,p}k^{\alpha+1}(j).
\end{eqnarray*}

Again, applying a inequality of Minkwoski type to the first addend of (\ref{ineq}) we have
\begin{eqnarray*}
    &&\left(\sum_{n=0}^\infty (k^{\alpha+1}(n))^p\left(\sum_{j=0}^n |W^\alpha g(j)|\sum_{u=n-j}^n k^\alpha(u-n+j)|W^\alpha f(u)|\right)^p \right)^{1/p}\\
    &\le& \sum_{j=0}^\infty |W^\alpha g(j)|\left(\sum_{n=j}^\infty(k^{\alpha+1}(n))^p \left(\sum_{u=n-j}^n k^\alpha(u-n+j)|W^\alpha f(u)|\right)^p \right)^{1/p}\\
    &\le& \sum_{j=0}^\infty |W^\alpha g(j)|\left[\left(\sum_{n=j}^{2j}(k^{\alpha+1}(n))^p \left(\sum_{u=n-j}^n k^\alpha(u-n+j)|W^\alpha f(u)|\right)^p \right)^{1/p}\right.\\
    &&+ \left.\left(\sum_{n=2j+1}^{\infty}(k^{\alpha+1}(n))^p \left(\sum_{u=n-j}^n k^\alpha(u-n+j)|W^\alpha f(u)|\right)^p \right)^{1/p}\right].
\end{eqnarray*}
For $j\le n\le 2j$ we write
\begin{eqnarray*}
    &&\left(\sum_{n=j}^{2j}(k^{\alpha+1}(n))^p \left(\sum_{u=n-j}^n k^\alpha(u-n+j)|W^\alpha f(u)|\right)^p\right)^{1/p}\\
    &=& k^{\alpha+1}(2j)\left(\sum_{l=0}^{j} \left(\sum_{u=l}^{l+j} k^\alpha(u-l)|W^\alpha f(u)|\right)^p\right)^{1/p}\\
    &\le& k^{\alpha+1}(2j)\left(\left(\sum_{u=0}^\infty k^\alpha(u)|W^\alpha f(u)|\right)^p+\sum_{l=1}^{\infty} \left(\sum_{u=l}^\infty k^\alpha(u-l)|W^\alpha f(u)|\right)^p\right)^{1/p}\\
    &\le& C_\alpha k^{\alpha+1}(j)\left(\left(\sum_{u=0}^\infty \frac{k^\alpha(u)}{k^{\alpha+1}(u)}k^{\alpha+1}(u)|W^\alpha f(u)|\right)^p\right.\\
    &&\left.+C_{\alpha,p}\sum_{l=1}^{\infty} \left(\sum_{u=l}^\infty \frac{(u-l)^{\alpha-1}}{u^{\alpha}}u^{\alpha}|W^\alpha f(u)|\right)^p\right)^{1/p}\\
    &\le&  C_{\alpha,p} k^{\alpha+1}(j)\|f\|_{\alpha,p},
\end{eqnarray*}
by using \cite[Lemma 2.1]{ALMV}, H\"older's inequality, \eqref{double} and Hardy's inequality \eqref{hardyDual}. For $n>2j$,
\begin{eqnarray*}
    &&\left(\sum_{n=2j+1}^{\infty}(k^{\alpha+1}(n))^p \left(\sum_{u=n-j}^n k^\alpha(u-n+j)|W^\alpha f(u)|\right)^p\right)^{1/p}\\
    &=& \left(\sum_{l=j+1}^{\infty}(k^{\alpha+1}(l+j))^p \left(\sum_{u=l}^{l+j} k^\alpha(u-l)|W^\alpha f(u)|\right)^p\right)^{1/p}\\
    &\le& \left(\sum_{l=j+1}^{\infty}(k^{\alpha+1}(2l))^p \left(\sum_{u=l}^{l+j} k^\alpha(u-l)|W^\alpha f(u)|\right)^p\right)^{1/p}\\
    &\le& C_\alpha\left(  \sum_{l=j+1}^{\infty}(k^{\alpha+1}(l))^p \left(\sum_{m=0}^{j} k^\alpha(m)|W^\alpha f(m+l)|\right)^p\right)^{1/p}\\
    &\le& C_\alpha \sum_{m=0}^{j}k^\alpha(m) \left(\sum_{l=j+1}^{\infty} (k^{\alpha+1}(l))^p|W^\alpha f(m+l)|^p\right)^{1/p}\\
    &\le& C_\alpha  \sum_{m=0}^{j}k^\alpha(m) \left(\sum_{l=j+1}^{\infty} (k^{\alpha+1}(m+l))^p|W^\alpha f(m+l)|^p\right)^{1/p}\\
    &\le& C_\alpha\|f\|_{\alpha,p} k^{\alpha+1}(j),
\end{eqnarray*}
by using \cite[Lemma 2.1]{ALMV} and a Minkwoski type inequality. So, we conclude the result.
\end{proof}

\section{$C_0$-Semigroups on $\tau_p^{\alpha}$}\label{sect5}

In this section we study, for $1\leq p\leq \infty,$ the $C_0$-semigroup structure of the two one-parameter operator families on the sequence spaces $\tau^{\alpha}_p$, $(T_p(t))_{t\geq 0}$ and $(S_p(t))_{t\geq 0},$ given by $$T_p(t)f(n):=e^{-\frac{t}{p}}T(t)f(n),\quad S_p(t)f(n):=e^{-t(1-\frac{1}{p})}S(t)f(n),\quad  n\in\N_0.$$

\begin{theorem} Take $1\leq p\leq\infty$ and $\alpha\geq 0.$ The one-parameter operator families $(T_p(t))_{t\geq 0}$ and $(S_p(t))_{t\geq 0}$ are contraction $C_0$-semigroups on $\tau_p^{\alpha},$ whose generators $A$ and $B$ respectively are given by $$Af(0):=-\frac{1}{p}f(0),\ Af(n):=-n\Delta f(n-1)-\frac{1}{p}f(n),\quad n\in \N,$$ $$Bf(n):=(n+1)\Delta f(n)+\frac{1}{p}f(n),\quad n\in\N_0,$$ with $f\in D(A)=D(B)=\tau_p^{\alpha+1}.$

The $C_0$-semigroup $(S_p(t))_{t\geq 0}$ intertwines with the operator $D^\alpha$ given in Theorem \ref{propiedades} (i), i.e., $S_p(t) \circ D^\alpha= D^{\alpha} \circ S_p(t)$ for $t\ge 0$.
\end{theorem}
\begin{proof}
First of all we prove that $T_1(t)$ is a bounded operator on $\ell^1$ and $\ell^{\infty}.$ Let $f\in\ell^1$ and $t>0,$ then \begin{eqnarray*}
\displaystyle\sum_{n=0}^{\infty}|T_1(t)f(n)| &\leq & e^{-t}\displaystyle\sum_{j=0}^{\infty}e^{-tj}|f(j)|\sum_{n=j}^{\infty}\binom{n}{j}(1-e^{-t})^{n-j} \\
&=& e^{-t}\displaystyle\sum_{j=0}^{\infty}e^{-tj}|f(j)|\sum_{n=0}^{\infty}\binom{n+j}{j}(1-e^{-t})^{n}=\lVert f\rVert_1,
\end{eqnarray*}
where we have used \eqref{generating} (the case $t=0$ is easy). Now let $f\in \ell^{\infty},$ then $$\sup_{n\in\N_0}|T_1f(n)|\leq e^{-t}\lVert f\rVert_{\infty}\sup_{n\in\N_0}\displaystyle\sum_{j=0}^n\binom{n}{j}e^{-tj}(1-e^{-t})^{n-j}=e^{-t}\lVert f\rVert_{\infty}.$$ Applying the Riesz-Thorin Theorem we get that $\lVert T_1(t)\rVert_{p}\leq e^{-t\frac{p-1}{p}},$ and consequently $\lVert T_p(t)\rVert_{p}\leq 1.$

Similarly, for $\alpha>0,$ let $f\in\tau_1^{\alpha},$ then \begin{displaymath}\begin{array}{l}
\displaystyle\sum_{n=0}^{\infty}|W^{\alpha}T_1(t)f(n)|k^{\alpha+1}(n)\leq\displaystyle\sum_{n=0}^{\infty}e^{-t}k^{\alpha+1}(n)\displaystyle\sum_{j=0}^{n}\binom{n}{j}e^{-t(j+\alpha)}(1-e^{-t})^{n-j}|W^{\alpha}f(j)|\\
=\displaystyle e^{-t}\sum_{j=0}^{\infty} |W^{\alpha}f(j)|e^{-t(j+\alpha)}k^{\alpha+1}(j) \sum_{n=j}^{\infty}\binom{n+\alpha}{j+\alpha}(1-e^{-t})^{n}\\
=\displaystyle e^{-t}\sum_{j=0}^{\infty} |W^{\alpha}f(j)|e^{-t(j+\alpha)} k^{\alpha+1}(j) \sum_{n=0}^{\infty}\binom{n+j+\alpha}{j+\alpha}(1-e^{-t})^{n}=\lVert f \rVert_{\alpha,1},
\end{array}\end{displaymath}
where we have applied Lemma \ref{WeylSemT}, and the identities $k^{\alpha+1}(n)\binom{n}{j}=k^{\alpha+1}(j)\binom{n+\alpha}{j+\alpha}$ and \eqref{generating} (the case $t=0$ is easy). Now let $f\in \tau_{\infty}^{\alpha}$ and $t>0,$ then \begin{displaymath}\begin{array}{l}
\sup_{n\in\N_0}|W^{\alpha}T_1(t)f(n)|k^{\alpha+1}(n) = \sup_{n\in\N_0}e^{-t}|\displaystyle\sum_{j=0}^{n}\binom{n}{j}e^{-t(j+\alpha)}(1-e^{-t})^{n-j}|W^{\alpha}f(j)||k^{\alpha+1}(n)\\
\leq\lVert f\rVert_{\alpha,\infty}\sup_{n\in\N_0}e^{-t}\displaystyle\sum_{j=0}^{n}\binom{n+\alpha}{j+\alpha}e^{-t(j+\alpha)}(1-e^{-t})^{n-j}= e^{-t}\lVert f\rVert_{\alpha,\infty},\end{array}\end{displaymath} where we have used again Lemma \ref{WeylSemT}, and the equality $k^{\alpha+1}(n)\binom{n}{j}=k^{\alpha+1}(j)\binom{n+\alpha}{j+\alpha}$. So, we have that $T_1(t)$ is a bounded operator on $\tau_1^{\alpha}$ and $\tau_{\infty}^{\alpha},$ with $$\lVert T_1(t)\rVert_{\alpha,1}\leq 1,\quad \lVert T_1(t)\rVert_{\alpha,\infty}\leq e^{-t}.$$
Using Theorem \ref{propiedades} (i), $D^{\alpha}T_1(D^{\alpha})^{-1}$ is a bounded operator on $\ell^1$ and $\ell^{\infty}$ with $$\lVert D^{\alpha}T_1(t)(D^{\alpha})^{-1}\rVert_1\leq 1,\quad \lVert D^{\alpha}T_1(t)(D^{\alpha})^{-1}\rVert_{\infty}\leq e^{-t}.$$ Applying the Riesz-Thorin theorem we get that $\lVert D^{\alpha}T_1(t)(D^{\alpha})^{-1}\rVert_{p}\leq e^{-t\frac{p-1}{p}},$ and therefore $T_p(t)$ is a bounded operator on $\tau_p^{\alpha}$ with $$\lVert T_p(t)\rVert_{\alpha,p}\leq 1.$$

On the other hand, we will prove that $S_1(t)$ is a bounded operator on $\ell^1$ and $\ell^{\infty}.$ Let $f\in \ell^1,$ then $$
\displaystyle\sum_{n=0}^{\infty}|S_1(t)f(n)| \leq  \displaystyle\sum_{j=0}^{\infty}|f(j)|\sum_{n=0}^{j}\binom{j}{n}e^{-tn}(1-e^{-t})^{j-n}=\lVert f\rVert_1.$$ Now let $f\in \ell^{\infty}$ and $t>0,$ by \eqref{generating} one gets \begin{eqnarray*}
\sup_{n\in\N_0}|S_1(t)f(n)| &\leq & \lVert f\rVert_{\infty}\sup_{n\in\N_0}e^{-tn}\displaystyle\sum_{j=n}^{\infty}\binom{j}{n}(1-e^{-t})^{j-n} \\
&=& \lVert f\rVert_{\infty}\sup_{n\in\N_0}e^{-tn}\displaystyle\sum_{j=0}^{\infty}\binom{j+n}{n}(1-e^{-t})^{j}=e^{t}\lVert f\rVert_{\infty}.\end{eqnarray*} The case $t=0$ is clear. Applying the Riesz-Thorin theorem we get that $\lVert S_1(t)\rVert_{p}\leq e^{t\frac{p-1}{p}},$ and consequently $\lVert S_p(t)\rVert_{p}\leq 1.$

Similarly, for $\alpha>0,$ let $f\in\tau_1^{\alpha},$ then \begin{eqnarray*}
\displaystyle\sum_{n=0}^{\infty}|W^{\alpha}S_1(t)f(n)|k^{\alpha+1}(n)&\leq&\displaystyle\sum_{n=0}^{\infty}e^{-tn}k^{\alpha+1}(n)\displaystyle\sum_{j=n}^{\infty}\binom{j+\alpha}{n+\alpha}(1-e^{-t})^{j-n}|W^{\alpha}f(j)|\\
&=&\displaystyle\sum_{j=0}^{\infty} |W^{\alpha}f(j)| k^{\alpha+1}(j) \sum_{n=0}^{j}e^{-tn}(1-e^{-t})^{j-n}\binom{j}{n}=\lVert f \rVert_{\alpha,1},
\end{eqnarray*}
where we have applied Lemma \ref{WeylSem} and the equality $k^{\alpha+1}(n)\binom{j+\alpha}{n+\alpha}=k^{\alpha+1}(j)\binom{j}{n}.$ Now let $f\in \tau_{\infty}^{\alpha}$ and $t>0,$ then \begin{eqnarray*}
\sup_{n\in\N_0}|W^{\alpha}S_1(t)f(n)|k^{\alpha+1}(n) &=& \sup_{n\in\N_0}e^{-tn}|\displaystyle\sum_{j=n}^{\infty}\binom{j+\alpha}{n+\alpha}(1-e^{-t})^{j-n}|W^{\alpha}f(j)||k^{\alpha+1}(n)\\
&\leq&\lVert f\rVert_{\alpha,\infty}\sup_{n\in\N_0}e^{-tn}\displaystyle\sum_{j=n}^{\infty}\binom{j}{n}(1-e^{-t})^{j-n}= e^t\lVert f\rVert_{\alpha,\infty},\end{eqnarray*} where we have used again Lemma \ref{WeylSem}, and the identities $k^{\alpha+1}(n)\binom{j+\alpha}{n+\alpha}=k^{\alpha+1}(j)\binom{j}{n}$ and \eqref{generating} (the case $t=0$ is easy). So, we have that $S_1(t)$ is a bounded operator on $\tau_1^{\alpha}$ and $\tau_{\infty}^{\alpha}$ with $$\lVert S_1(t)\rVert_{\alpha,1}\leq 1,\quad \lVert S_1(t)\rVert_{\alpha,\infty}\leq e^{t}.$$ By Theorem \ref{propiedades} (i), $D^{\alpha}S_1(D^{\alpha})^{-1}$ is a bounded operator on $\ell^1$ and $\ell^{\infty}$ with $$\lVert D^{\alpha}S_1(t)(D^{\alpha})^{-1}\rVert_1\leq 1,\quad \lVert D^{\alpha}S_1(t)(D^{\alpha})^{-1}\rVert_{\infty}\leq e^{t}.$$ Applying the Riesz-Thorin theorem we get that $\lVert D^{\alpha}S_1(t)(D^{\alpha})^{-1}\rVert_{p}\leq e^{t\frac{p-1}{p}},$ and therefore $S_p(t)$ is a bounded operator on $\tau_p^{\alpha}$ with $$\lVert S_p(t)\rVert_{\alpha,p}\leq 1.$$

The strong continuity follows from the ideas developed in \cite[Section 4, Example 7.1 and 7.2]{Sis}, using properties of the operator $W^{\alpha}$ (\cite[Section 2]{ALMV}). It is a simple check that $A$ and $B$ are the generators.

Finally we prove that $D(A)=D(B)=\tau_p^{\alpha+1}.$ Let $f\in \tau_{p}^{\alpha+1}$ be given, then $f\in \tau_{p}^{\alpha}$ since $\tau_{p}^{\alpha+1}\hookrightarrow\tau_p^{\alpha}.$ Using Proposition \ref{leibniz} we have that \begin{equation}\label{leibniz2}W^{\alpha}((j+1)\Delta f(j))(n)=-(n+\alpha+1)W^{\alpha+1}f(n)+\alpha W^{\alpha}f(n),\quad n\in \N_0,\end{equation} then $(j+1)\Delta f(j)\in \tau_p^{\alpha}$ and therefore $f\in D(B).$ Conversely, if $f\in D(B),$ then $f\in \tau_p^{\alpha}$ and $(j+1)\Delta f(j)\in \tau_p^{\alpha}.$ The identity \eqref{leibniz2} implies that $$(\alpha+2)W^{\alpha+1}f(n)k^{\alpha+2}(n)=-W^{\alpha}((j+1)\Delta f(j))(n)k^{\alpha+1}(n)+\alpha W^{\alpha}f(n)k^{\alpha+1}(n),\quad n\in \N_0,$$ whence $f\in \tau_p^{\alpha+1}.$ The case $D(A)=\tau_p^{\alpha+1}$ is similar.

The equality $S_p(t) \circ D^\alpha= D^{\alpha} \circ S_p(t)$ for $t\ge 0$ is a direct consequence of Lemma \ref{WeylSem} and the definition of the $C_0$-semigroup $(S_p(t))_{t\ge 0}$.
\end{proof}

\begin{remark}\label{holomor}{\rm The above semigroups are not holomorphic. First we see that $(T_{p}(t))_{t\geq 0}$ are not holomorphic: we take $k^0,$ which belongs to $\tau_p^{\alpha}$ for all $\alpha\geq 0$ and $1\leq p\leq \infty.$ For $z\in\C, $ it is easy to see that $$\sup_{n\in\N_0}|W^{\alpha}T_1(z)k^0(n)|k^{\alpha+1}(n)=e^{-(\alpha+1)\Real z}\sup_{n\in\N_0}|1-e^{-z}|^n k^{\alpha+1}(n).$$ In the following we prove that there is not $0<\delta<\frac{\pi}{2}$ such that $|1-e^{-z}|\leq 1$ on $\Sigma_{\delta}:=\{z\in\C\,:\,z\neq 0\text{ and }|\text{Arg}z|<\delta\}.$ Then $(T_1(t))_{t\geq 0}$ is not holomorphic on $\tau_{\infty}^{\alpha},$ and therefore not on $\tau_p^{\alpha}$ for $1\leq p\leq \infty.$ Observe that $$|1-e^{-z}|=\sqrt{1-2e^{-\Real z}\cos(\Imag z)+e^{-2\Real z}}\leq 1$$ if and only if $$2\cos(\Imag z)-e^{-\Real z}\geq 0.$$ For simplicity, we write $z=\rho e^{i\varphi}$ with $0\leq \rho$ and $|\varphi|\leq \delta,$ and we define $$F(\rho,\varphi)=2\cos(\rho \sin(\varphi))-e^{-\rho\cos(\varphi)}.$$ Then for all $0<\delta_0<\delta$ and $\rho_0=\frac{\pi}{2\sin(\delta_0)}$ one gets $$F(\rho_0,\delta_0)=-e^{-\frac{\pi\cos(\delta_0)}{2\sin(\delta_0)}}<0,$$ so $|1-e^{-z}|>1$ in some points of $\Sigma_{\delta}$ for any $0<\delta<\frac{\pi}{2}.$
\\

Secondly, we see that for all $0<\delta<\frac{\pi}{2},$ there exist points $z\in \Sigma_{\delta}$ such that $S_{p}(z)$ is not defined on $\tau_p^{\alpha}.$ It is suffices to take $z=\frac{\pi}{2\sin(\delta_0)}e^{i\delta_0}$ with $0<\delta_0<\delta.$ Then, the sequence $f(n)=\frac{1}{(1-e^{-z})^n},$ $n\in\N_0,$  belongs to $\tau_p^{\alpha}$ since $|1-e^{-z}|>1,$ but it is easy to see that $$S_p(z)f(n)=e^{-z(n+1-\frac{1}{p})}(1-e^{-z})^n\displaystyle\sum_{j=n}^{\infty}\binom{j}{n}=\infty.$$
}
\end{remark}

The proof of the following result it follows easily from Proposition \ref{propiedades}(iv).

\begin{proposition}\label{semigroups} The semigroups $(T_p(t))_{t\geq 0}$ and $(S_{p'}(t))_{t\geq 0}$ are dual operators of each other on $\tau_p^{\alpha}$ and $\tau_{p'}^{\alpha}$ respectively with $\frac{1}{p}+\frac{1}{p'}=1.$
\end{proposition}

\begin{proposition} Let $A$ and $B$ the generators of $(T_p(t))_{t\geq 0}$ and $(S_{p}(t))_{t\geq 0}$ on $\tau_p^{\alpha}$ with $1\leq p\leq\infty.$ Then \begin{enumerate}
\item[(i)] The point spectrum of $A$ and $B$ are:\begin{itemize}
\item[(a)] $\sigma_{point}(A)=\emptyset$ and $\sigma_{point}(B)=\C_-,$ for $1< p<\infty.$
\item[(b)] $\sigma_{point}(A)=\emptyset$ and $\sigma_{point}(B)=\C_-\cup\{0\},$ for $p=1.$
\item[(c)] $\sigma_{point}(A)=\{0\}$ and $\sigma_{point}(B)=\C_-\cup i\R,$ for $p=\infty.$
\end{itemize}
\item[(ii)] The spectrum of $B$ is $\sigma(B)=\C_-\cup i\R.$
\end{enumerate}
\end{proposition}
\begin{proof}(i) First we consider the operator $A$. For the case (a) and (b), we take $1\leq p<\infty.$ Let $\lambda\in\C$ and $f\in \tau_p^{\alpha+1}$ such that $Af=\lambda f.$ Then the nonzero solutions of this difference equation are $f(n)=c$ with $c\neq 0,$ which do not belong to $\tau_p^{\alpha+1}.$ For the case (c), let $p=\infty.$ If $\lambda\neq 0,$ $Af=\lambda f$ if and only if $f$ is the null sequence. On the other hand, if $\lambda=0,$ the nonzero solutions of $Af=\lambda f$ are $f(n)=c$ with $c\neq 0,$ which belong to $\tau_{\infty}^{\alpha+1},$ see Lemma \ref{belong-p}(i).

Now let $Bf=\lambda f.$ The nonzero solutions of this difference equation are $f(n)=c k^{\lambda+1-\frac{1}{p}}(n)$ with $c\neq 0,$ which are on $\tau_p^{\alpha+1}$ for $\Real\lambda <0$ if $1< p\leq \infty,$ for $\Real\lambda <0$ and $\lambda=0$ for $p=1,$ and for $\Real \lambda \leq 0$ if $p=\infty,$ see Lemma \ref{belong-p}(i).

(ii) We know that $(S_p(t))_{t\geq 0}$ is a contraction $C_0$-semigroup on $\tau_p^{\alpha}$, then by the spectral mapping theorem (\cite[Chapter IV, Theorem 3.6]{Nagel}) we have $$e^{t\sigma(B)}\subseteq \sigma(S_p(t))\subseteq \{z\in \C\,:\, |z|\leq 1\},$$ therefore $\sigma(B)\subseteq \{\Real\lambda \leq 0\}.$ Conversely, note that $$\sigma_{point}(B)\subseteq \sigma(B)\subseteq \{\Real\lambda \leq 0\}.$$ For $p=\infty$ we conclude the result. If $1\leq p<\infty,$ then it suffices to prove that the points $\{\Real\lambda =0\}$ are in $\sigma(B).$ Let $\mu \in i\R$ and we suppose that $\mu\in \rho(A).$ Let $\xi= \mu-\frac{1}{p}+2$ and $f\in \tau_p^{\alpha}.$ Using that $(\mu-B)^{-1}$ is a bounded operator, the sequence $g:=(\mu-B)^{-1}f\in \tau_p^{\alpha}.$ So $g$ is the solution of $g(n+1)=\frac{(\xi+n-1)g(n)-f(n)}{n+1}$ for $n\in\N_0.$ It easy to prove that $g(n+1)=ck^{\xi-1}(n+1)-\displaystyle\sum_{j=0}^n\frac{\Gamma(\xi+n)(n-j)!f(n-j)}{\Gamma(\xi+n-j)(n+1)!}$ where $c$ is a constant. If we take $f(n)=0$ for all $n\in\N_0$ and $c=1,$ then $g=k^{\xi-1}\not\in\tau_p^{\alpha}$ for $1<p<\infty,$ and $g=k^{\xi-1}\not\in\tau_1^{\alpha}$ except for the case $\xi=1$ (this last case is not relevant because for $p=1$ and $\xi=1$ we have $\mu =0,$ and we know that $0\in\sigma_{point(B)}$), see Lemma \ref{belong-p}.
\end{proof}

\begin{remark}{\rm The operator $-B$ is sectorial of angle $\pi/2$ in the sense of \cite[Chapter 2, Section 2.1]{Haase}.
}
\end{remark}

\section{Ces\`aro discrete operators on spaces defined on $\N_0$}
\label{dis}

In this section, we expose the results which contain the main aim of the paper. Let $\Real\beta>0,$ we consider the Ces\`aro operator of order $\beta$ given by $$\mathcal{C}_{\beta}f(n)=\frac{1}{k^{\beta+1}(n)}\sum_{j=0}^nk^{\beta}(n-j)f(j)\quad n\in\N_0,$$ and the dual Ces\`aro operator of order $\beta$ given by $$\mathcal{C}^*_{\beta}f(n)=\sum_{j=n}^{\infty}\frac{1}{k^{\beta+1}(j)}k^{\beta}(j-n)f(j)\quad n\in\N_0.$$

\begin{remark}{\rm \begin{itemize}
\item[(i)]Let $|\lambda|>1.$ We consider the sequence $r_{\lambda}(n)=\lambda^{-(n+1)},$ $n\in\N_0.$ Note that $$\mathcal{C}_{1}r_{\lambda}(n)=\frac{1-r_{\lambda}(n)}{(n+1)(\lambda-1)},\quad \mathcal{C}_{2}r_{\lambda}(n)=\frac{2((n+1)(\lambda-1)-1+r_{\lambda}(n))}{(n+1)(n+2)(\lambda-1)^2},\quad n\in\N_0.$$ On the other hand, $$\mathcal{C}_{1}^2r_{\lambda}(n)=\frac{1}{(n+1)(\lambda-1)}\sum_{j=0}^n \frac{1-r_{\lambda}(j)}{j+1},\quad n\in\N_0.$$ Then $\mathcal{C}_{1}^2 r_{\lambda}\neq\mathcal{C}_{2}r_{\lambda}$ (it suffices to evaluate the above sequence in $n=0$). Then $\mathcal{C}_{1}^2\neq \mathcal{C}_{2}.$
    \\

\item[(ii)] Let $p=1,$ $\Real\beta>0,$ $\alpha\geq 0$ and $f(n)=k^{-\beta}(n+1),$ $n\in\N_0.$ Note that $$\mathcal{C}_{\beta}f(n)=\frac{1}{k^{\beta+1}(n)}\sum_{j=1}^{n+1}k^{\beta}(n+1-j)k^{-\beta}(j)=-\frac{k^{\beta}(n+1)}{k^{\beta+1}(n)}=-\frac{\beta}{n+1},\quad n\in \N_0,$$ where we have used that $k^{\beta}*k^{-\beta}(n+1)=k^0(n+1)=0$ for $n\in\N_0.$ The sequence $f$ belongs to $\tau^{\alpha}_1$ (Lemma \ref{belong-p}(i)) and $\mathcal{C}_{\beta}f\notin \ell^1,$ therefore $\mathcal{C}_{\beta}$ is not a bounded operator on $\tau_1^{\alpha}.$\\

\item[(iii)] Let $p=\infty,$ $\Real\beta>0,$ $\alpha\geq 0$ and $f(n)=1$ for all $n\in\N_0.$ By Lemma \ref{belong-p}(i), the sequence $f\in\tau_{\infty}^{\alpha}.$ The equivalence \eqref{double2} implies that $\mathcal{C}_{\beta}^{*}f(n)$ diverges for each $n\in \N_0.$ So, $\mathcal{C}_{\beta}^{*}$ is not a bounded operator on $\tau_{\infty}^{\alpha}.$

\end{itemize}
}
\end{remark}

\begin{theorem}\label{bounded} Let $\alpha\geq 0$ and $\Real\beta>0.$ Then \begin{itemize}
\item[(i)] The operator $\mathcal{C_{\beta}}$ is a bounded operator on $\tau_p^{\alpha},$ for $1<p\leq\infty.$
\item[(ii)] The operator $\mathcal{C_{\beta}^*}$ is a bounded operator on $\tau_p^{\alpha},$ for $1\leq p <\infty.$
\end{itemize} In addition, for $f\in \tau_p^{\alpha}$ the following subordination identities hold, $$\mathcal{C}_{\beta}f(n)=\displaystyle\beta\int_0^{\infty}(1-e^{-t})^{\beta-1}e^{-t(1-\frac{1}{p})}T_p(t)f(n)\,dt,\quad n\in\N_0,\, 1<p\leq\infty,$$ and $$\mathcal{C}^*_{\beta}f(n)=\beta\int_0^{\infty}(1-e^{-t})^{\beta-1}e^{-\frac{t}{p}}S_p(t)f(n)\,dt,\quad n\in\N_0,\, 1\leq p <\infty.$$
\end{theorem}
\begin{proof} Let $\alpha\geq 0,$ $\Real\beta>0$ and $f\in \tau_p^{\alpha}.$ Then \begin{eqnarray*}
\mathcal{C}_{\beta}f(n)&=&\frac{1}{k^{\beta+1}(n)}\displaystyle\sum_{j=0}^{n}k^{\beta}(n-j)f(j)=\beta\sum_{j=0}^{n}\binom{n}{j}\frac{\Gamma(j+1)\Gamma(n+\beta-j)}{\Gamma(n+\beta+1)}f(j) \\
&=&\displaystyle\beta\sum_{j=0}^{n}\binom{n}{j}f(j)\int_0^1(1-x)^{n+\beta-j-1}x^{j}\,dx\\
&=&\beta\sum_{j=0}^{n}\binom{n}{j}f(j)\int_0^{\infty}(1-e^{-t})^{n+\beta-j-1}e^{-t(j+1)}\,dt \\
&=&\displaystyle\beta\int_0^{\infty}(1-e^{-t})^{\beta-1}e^{-t(1-\frac{1}{p})}T_p(t)f(n)\,dt,
\end{eqnarray*}
and \begin{eqnarray*}
\mathcal{C}_{\beta}^*f(n)&=&\displaystyle\sum_{j=n}^{\infty}\frac{1}{k^{\beta+1}(j)}k^{\beta}(j-n)f(j)=\beta\sum_{j=n}^{\infty}\binom{j}{n}\frac{\Gamma(n+1)\Gamma(j+\beta-n)}{\Gamma(j+\beta+1)}f(j) \\
&=&\displaystyle\beta\sum_{j=n}^{\infty}\binom{j}{n}f(j)\int_0^1(1-x)^{j+\beta-n-1}x^{n}\,dx\\
&=&\beta\sum_{j=n}^{\infty}\binom{j}{n}f(j)\int_0^{\infty}(1-e^{-t})^{j+\beta-n-1}e^{-t(n+1)}\,dt \\
&=&\displaystyle\beta\int_0^{\infty}(1-e^{-t})^{\beta-1}e^{-\frac{t}{p}}S_p(t)f(n)\,dt.
\end{eqnarray*}
Using these identities and the contractivity of the semigroups we get \begin{eqnarray*}
\lVert \mathcal{C}_{\beta}f\rVert_{\alpha,p}&\leq&|\beta|\displaystyle\lVert f\rVert_{\alpha,p}\int_0^{\infty}(1-e^{-t})^{\Real \beta-1}e^{-t(1-\frac{1}{p})}\,dt\\
&=&|\beta|\frac{\Gamma(\Real\beta)\Gamma(1-\frac{1}{p})}{\Gamma(\Real \beta+1-\frac{1}{p})}\lVert f\rVert_{\alpha,p},\quad 1<p\leq\infty,
\end{eqnarray*}
and
\begin{eqnarray*}
\lVert \mathcal{C}_{\beta}^* f\rVert_{\alpha,p}&\leq&|\beta|\displaystyle\lVert f\rVert_{\alpha,p}\int_0^{\infty}(1-e^{-t})^{\Real\beta-1}e^{-\frac{t}{p}}\,dt\\
&=&|\beta|\frac{\Gamma(\Real\beta)\Gamma(\frac{1}{p})}{\Gamma(\Real \beta+\frac{1}{p})}\lVert f\rVert_{\alpha,p},\quad 1\leq p<\infty,
\end{eqnarray*}
and we conclude the proof
\end{proof}

\begin{remark}\label{holo}{\rm Let $H^p$ be the usual Hardy space on the unit disc $\D$ for $1\le p\le\infty$. The generalized Ces\`{a}ro operators (or $\beta$-Ces\`{a}ro operators) ${\frak C}_\beta$ and ${\frak C}_\beta^*$ are defined by
\begin{eqnarray*}
{\frak C}_\beta(F)(z)&:=&{\beta\over z^{\beta}}\int_0^z F(w){(z-w)^{\beta-1}\over (1-w)^{\beta}}dw, \qquad z \in \D,\cr
{\frak C}_\beta^*(F)(z)&:=&{\beta\over (z-1)^{\beta}}\int_1^z F(w){(z-w)^{\beta-1}}dw, \qquad z \in \D.\cr
\end{eqnarray*}
 Then ${\frak C}_\beta$ is a bounded operator for $1\le p<\infty$ and  ${\frak C}_\beta^*$ is a bounded operator for $1< p\le \infty,$ see \cite{Xiao} and \cite{Stempak}. Moreover by \cite[Formula (3.3) and Formula (2.1)]{Xiao}, we have that
\begin{eqnarray*}
{\frak C}_\beta(F)(z)&=&{\beta}\int_0^\infty {e^{-t} \over 1-(1-e^{-t})z} F(\psi_t(z))(1-e^{-t})^{\beta-1}dt, \qquad z \in \D,\cr
{\frak C}_\beta^*(F)(z)&:=&{\beta}\int_0^\infty e^{-t} F(\phi_t(z)){(1-e^{-t})^{\beta-1}}dt, \qquad z \in \D,\cr
\end{eqnarray*}
where the composition semigroups $(\psi_t)_{t>0}$ and $(\phi_t)_{t>0}$ are considered in Remark \ref{compos}. In the case that $F=\tilde f$ (for suitable $f$) we get
 \begin{eqnarray*}
{\frak C}_\beta(F)(z)&=&{\beta}\int_0^\infty {e^{-t}} \widetilde{T(t)f}(z)(1-e^{-t})^{\beta-1}dt=\widetilde{{\mathcal C}_\beta(f)}(z), \qquad z \in \D,\cr
{\frak C}_\beta^*(F)(z)&:=&{\beta}\int_0^\infty e^{-t} \widetilde{S(t)f}(z){(1-e^{-t})^{\beta-1}}dt=\widetilde{{\mathcal C}_\beta^*(f)}(z), \qquad z \in \D.\cr
\end{eqnarray*}
}
\end{remark}

The following result states the duality between the generalized Ces\`aro operators. The proof is a simple check and we omit it.

\begin{proposition} Let $\alpha\geq 0$ and $\Real\beta>0.$ The Ces\`aro operators $\mathcal{C_{\beta}}$ and $\mathcal{C_{\beta}^*}$ are dual operators of each other $\tau_p^{\alpha}$ and $\tau_{p^{'}}^{\alpha}$ respectively with $\frac{1}{p}+\frac{1}{p^{'}}=1$ and $1<p\leq\infty.$
\end{proposition}

\begin{remark}\label{remark5.4}{\rm Let $f\in \tau_p^{\alpha}$  and $\beta=1.$ Then \begin{eqnarray*}
\mathcal{C}_{1}f(n)&=&\displaystyle\int_0^{\infty}e^{-t(1-{1\over p})}T_p(t)f(n)\,dt=(1-{1\over p}-A)^{-1}f(n), \quad 1<p\le \infty,\cr
\mathcal{C}_{1}^{*}f(n)&=&\displaystyle\int_0^{\infty}e^{-t/p}S_p(t)f(n)\,dt=({1\over p}-B)^{-1}f(n),\quad 1\le p< \infty \end{eqnarray*} for $n\in\N_0.$ By the spectral resolvent theorem for the resolvent operator (\cite[Chapter IV, Theorem 1.13]{Nagel}), one gets $$\sigma(\mathcal{C}_{1}^{*})=\left\{z\in \C\,:\, \arrowvert z-p/2\arrowvert\leq p/2 \right\}.$$
A  novel proof of  the equality $\sigma(\mathcal{C}_{1})=\left\{z\in \C\,:\, \arrowvert z-p'/2\arrowvert\leq p'/2 \right\}$ where the
Ces\`{a}ro operator $\mathcal{C}_{1}$ is working on $\ell^p,$ for $1<p<\infty,$ is given in \cite{curbera}. Also, in \cite{Bonet}, the authors study the spectrum of $\mathcal{C}_{1}$ on weighted $\ell^p$ spaces. Some of these results are particular cases of the following  theorem.
}
\end{remark}

\begin{theorem} \label{spectro} Let $\alpha\geq 0$ and $\Real\beta>0.$ Then \begin{itemize}
\item[(i)] The operator $\mathcal{C_{\beta}}:\tau_p^{\alpha}\to \tau_p^{\alpha}$ satisfies $\sigma(\mathcal{C_{\beta}})=\overline{\biggl\{ \frac{\Gamma(\beta+1)\Gamma(z+1-\frac{1}{p})}{\Gamma(\beta+z+1-\frac{1}{p})}\,:\, z\in\C_+\cup i\R \biggr\}},$ for $1<p\leq\infty.$
\item[(ii)] The operator $\mathcal{C_{\beta}}^*:\tau_p^{\alpha}\to \tau_p^{\alpha}$ satisfies $\sigma(\mathcal{C_{\beta}}^*)=\overline{\biggl\{ \frac{\Gamma(\beta+1)\Gamma(z+\frac{1}{p})}{\Gamma(\beta+z+\frac{1}{p})}\,:\, z\in\C_+\cup i\R \biggr\}},$ for $1\leq p <\infty.$
\end{itemize}
\end{theorem}
\begin{proof} (ii) Let $1\leq p <\infty.$ The family $(S_p(t))_{t\geq 0}$ is an uniformly bounded $C_0$-semigroup on $\tau_p^{\alpha}$ generated by $(B,D(B)).$ Then we consider the Hille-Phillips functional calculus $\mathcal{L}(\cdot)(-B):L^1(\R_+)\to\mathcal{B}(X)$ given by $$\mathcal{L}(h)(-B)f=\int_0^{\infty}h(t)S_p(t)f\,dt,\quad h\in L^1(\R_+),\,f\in \tau_p^{\alpha}.$$ By Theorem \ref{bounded}, for $1\leq p <\infty,$ we have that $\mathcal{C}_{\beta}^*=\mathcal{L}(h_{\beta,p})(-B),$ that is, $$\mathcal{C}_{\beta}^*f=\beta\int_0^{\infty}(1-e^{-t})^{\beta-1}e^{-\frac{t}{p}}S_p(t)f\,dt=\int_0^{\infty}h_{\beta,p}(t)S_p(t)f\,dt,$$ with $h_{\beta,p}(t):=\beta(1-e^{-t})^{\beta-1}e^{-\frac{t}{p}}$ for $t>0.$  Observe that $h_{\beta,p}\in L^1(\R_+),$ $$\mathcal{L}(h_{\beta,p})(z)=\frac{\Gamma(\beta+1)\Gamma(z+\frac{1}{p})}{\Gamma(\beta+z+\frac{1}{p})}:=g_{\beta,p}(z),\quad  z\in\overline{\C_+},\,1\leq p <\infty,$$ where $\mathcal{L}$ denotes the Laplace transform, and $\mathcal{L}(h_{\beta,p})\in C_0(\R_+)\cap\mathcal{H}_0(\C_+).$

Observe that we can extend in an holomorphic way the function $g_{\beta,p}$ to $\Sigma_{\varphi_0}:=\{z\in\C\,:\,z\neq 0\text{ and }|\text{Arg}z|<\varphi_0\},$ where $\varphi_0\in (\frac{\pi}{2},\pi-\varphi_1)$ with $\varphi_1:=\arctan(\frac{\Imag \beta}{\Real\beta})$ (see figure below): since $\Gamma$ is a meromorphic function with poles in $\{0,-1,-2,\cdots\},$ we could have problems of holomorphy in points $z\in\C_-\cap \Sigma_{\varphi_0}$ such that $\Real (\beta+z+1/p)\in\{0,-1,-2,\cdots\}$ and $\Imag(\beta+z+1/p)=\Imag(\beta+z)=0.$ It is easy to see that the above does not take place.

\vspace*{1.0cm}

\begin{center}

\begin{tikzpicture}
    \fill[gray!10,rounded corners]  (-2,3) -- (0,0) -- (4,0) -- (4,3) -- cycle;
    \fill[gray!10,rounded corners]  (-2,-3) -- (0,0) -- (4,0) -- (4,-3) -- cycle;
    \draw[dashed,gray](-4,0)--(4,0);
    \draw[dashed,gray](0,-3)--(0,3);

    \path (0,0)++(80:1cm)node{$\varphi_0$};
    \draw[->](0.75,0)arc(0:124:.75cm);

    \draw[gray](0:0cm)--(140:4cm);
    \path (0,0)++(165:2.25cm) node{$\pi-\varphi_1$};
    \draw[->] (-1.5,0)arc(180:140:1.5cm);

    \draw[gray](0:0cm)--(40:2.5cm);
    \path (0,0)++(15:1.75cm) node{$\varphi_1$};
    \draw[->] (1.5,0)arc(0:40:1.5cm);

   \fill[black] (1.9,1.6) circle (1.3pt);
   \draw[thick](2.0,1.7)node[right]{$\overline{\beta}$};

   \fill[black] (-0.6,1.6) circle (1.3pt);
   \draw[thick](-0.7,1.8)node[right]{$z$};

    \draw[black](-0.6,1.6)--(1.9,1.6);

    \draw[thick](2.8,2.5)node[right]{$\Sigma_{\varphi_0}$};
    %
\end{tikzpicture}

\end{center}

\vspace*{1.0cm}

In addition, note that $g_{\beta,p}(0)=\frac{\Gamma(\beta+1)\Gamma(\frac{1}{p})}{\Gamma(\beta+\frac{1}{p})}$ and $g_{\beta,p}$ is holomorphic in a neighborhood of $0,$ then $g_{\beta,p}$ has finite polynomial limit at $0.$ Also, by \eqref{double3},$\displaystyle\lim_{|z|\to\infty}g_{\beta,p}(z)=0$ with $z\in S_{\varphi_0},$ so $g_{\beta,p}$ has finite polynomial limit at $\infty.$ Using \cite[Lemma 2.2.3]{Haase} we have that $g_{\beta,p}\in\mathcal{E}_{\varphi_0},$ where $\mathcal{E}_{\varphi_0}$ denotes the extended Dunford-Riesz class.  Therefore, since $-B$ is a sectorial operator of angle $\frac{\pi}{2}$ and $B$ is injective ($0\notin \sigma_{p}(B)$) we can apply the spectral mapping theorem \cite[Theorem 2.7.8]{Haase} and we get $$\sigma(\mathcal{C}_{\beta}^*)=\sigma(g_{\beta,p}(-B))=\overline{g_{\beta,p}(\sigma(-B))}=\overline{\biggl\{ \frac{\Gamma(\beta+1)\Gamma(z+\frac{1}{p})}{\Gamma(\beta+z+\frac{1}{p})}\,:\, z\in\C_+\cup i\R \biggr\}}.$$
(i) The proof is consequence of the part (ii) and the duality between $\mathcal{C_{\beta}}$ and $\mathcal{C_{\beta}}^*$.
\end{proof}

\begin{remark}\label{remark5.6}{\rm Observe that for $p=2,$ we have $\sigma(\mathcal{C_{\beta}})=\sigma(\mathcal{C_{\beta}^*})$ for all $\beta>0.$ If $1<p\leq\infty,$ then $\sigma(\mathcal{C_{\beta}})$ (where $\mathcal{C_{\beta}}$ is considered on $\tau_p^{\alpha}$) is equal to $\sigma(\mathcal{C_{\beta}^*}),$ with $\mathcal{C_{\beta}^*}$ on $\tau_{p'}^{\alpha},$ with $\frac{1}{p}+\frac{1}{p'}=1.$
}
\end{remark}

\section{Spectral pictures, Final comments}

The main aim of this last section is that the reader visualizes the spectrum of the Ces\`aro operators on some particular cases. We will use Mathematica in order to draw the desired sets.

First of all, it is interesting to observe that the following pictures show the border in the values $\{z=it\ :\ t\in\R\}$ of the spectral sets, since it provides a better view of the behaviour of such sets. Note that this subset of the spectrum may not be the mathematical boundary as shown in the Figures. For each $1<p\leq \infty$ and $\beta>0,$ we consider the border of the spectrum of $\mathcal{C}_{\beta}$ on $\tau_p^{\alpha}$ as the curve $$\Gamma_{p,\beta}:=\overline{\biggl\{ \frac{\Gamma(\beta+1)\Gamma(it+1-\frac{1}{p})}{\Gamma(\beta+it+1-\frac{1}{p})}\,:\, t\in\R \biggr\}}.$$ Note that this curve degenerates in the point $1$ as $\beta\to 0^+.$ In \cite{LMPL}, the authors study the spectrum of the generalized Ces\`aro operators for the continuous case on Sobolev spaces. Precisely, their spectrums coincide with the curves $\Gamma_{p,\beta}$ in the discrete case.\\

As we mention in Remark \ref{remark5.4}, for $1\leq p<\infty,$ the spectrum of $\mathcal{C}_{1}^{*}$ on $\tau_p^{\alpha}$ is the closed ball centered in $p/2$ and radius $p/2.$ Therefore, we have that $$\sigma(\mathcal{C}_{1})=\left\{z\in \C\,:\, \arrowvert z-\frac{p}{2(p-1)}\arrowvert\leq \frac{p}{2(p-1)} \right\}$$ on $\tau_p^{\alpha}$ for $1<p\leq \infty,$ see Remark \ref{remark5.6}. This comment can be observed in the Figure 1 and Figure 3.\\

First, we will show cases where $\beta>0.$ For each $1<p\leq \infty$ and $\beta>0,$ the curve $\Gamma_{p,\beta}$ takes the point $\frac{\Gamma(\beta+1)\Gamma(1-1/p)}{\Gamma(\beta+1-1/p)}$ on the complex plane (doing $t=0$). In addition, when $p\neq \infty,$ the point $\frac{\Gamma(\beta+1)\Gamma(1-1/p)}{\Gamma(\beta+1-1/p)}\to\infty$ as $\beta\to \infty,$ see \eqref{double3}. This last comment is appreciated very well in the Figure 2. On the other hand, by \eqref{double3} $\frac{\Gamma(\beta+1)\Gamma(it+1-\frac{1}{p})}{\Gamma(\beta+it+1-\frac{1}{p})}\to 0$ as $t\to\pm\infty,$ for all $1<p\leq \infty$ and $\beta>0.$\\

The Figure 1 shows the curve $\Gamma_{2,\beta}$ for $\beta=0.5,1,2,3,5.$ It seems that for $\beta<1,$ $\Gamma_{2,\beta}$ is a curve contained in the circle centered in $1$ and radius $1.$ Such curve only cuts the real axis on $0$ and $\frac{\Gamma(\beta+1)\sqrt{\pi}}{\Gamma(\beta+1/2)}.$ We conjecture that the last statement is true for $0<\beta\leq 1$ and $1<p\leq \infty.$ On the other hand, for $\beta>1,$ $\Gamma_{2,\beta}$ has points on $\C_-.$ In fact, the curve cuts the imaginary axis at least in two points. Also, for $\beta=3$ and $5,$ the curve $\Gamma_{2,\beta}$ cuts the real negative axis. For $\beta=2,$ it can not be appreciated the last statement.

\begin{figure}[h]
\caption{}
\centering
\includegraphics[width=0.6\textwidth]{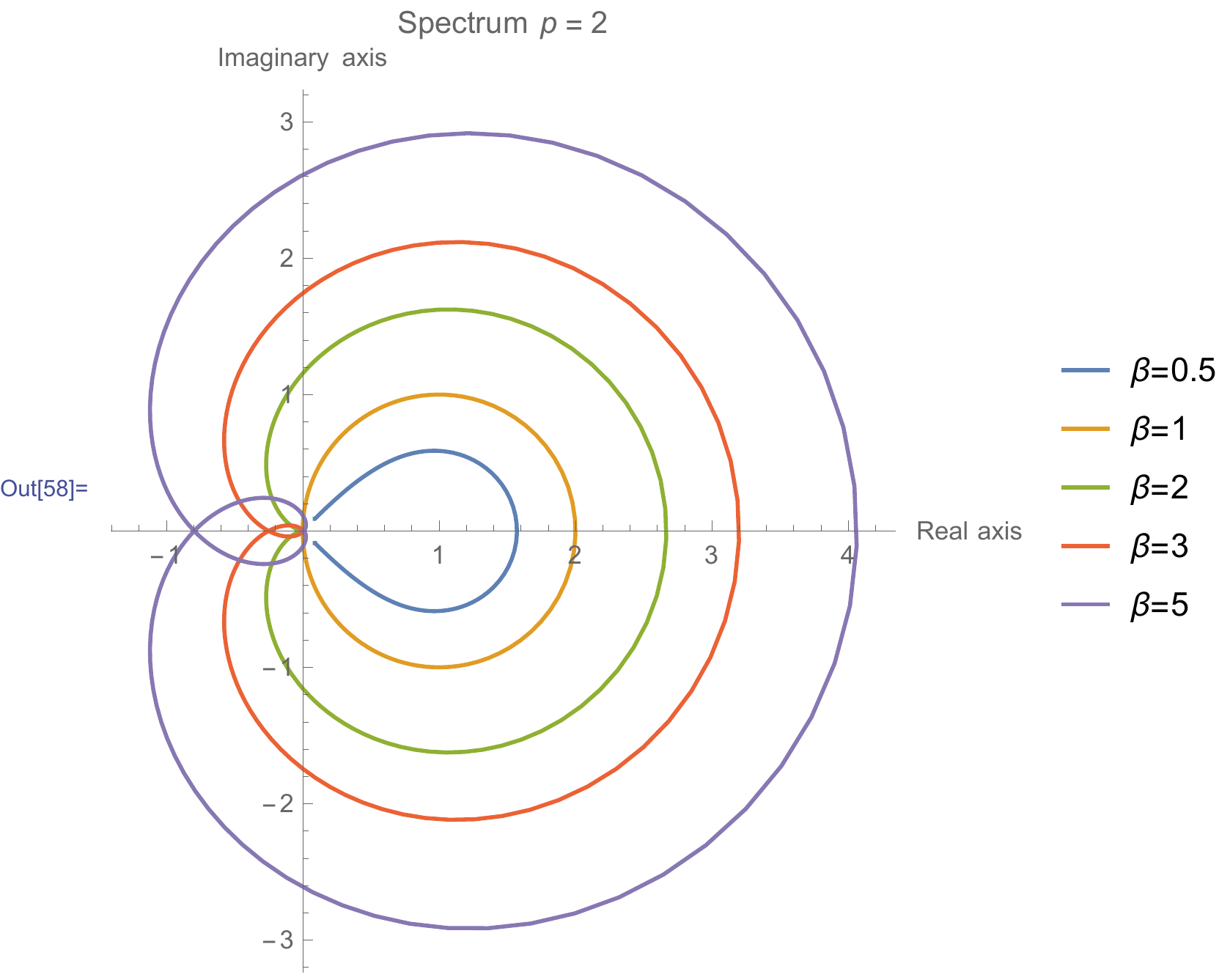}
\end{figure}


The Figure 2 shows the curve $\Gamma_{1.5,\beta}$ for $\beta=100$ and $\beta=200.$ For these values, it cuts several times the real and imaginary axes. It would be interesting to know such points, or at least, how many times each axis is cut.

\begin{figure}[h]
\caption{}
\centering
\includegraphics[width=0.6\textwidth]{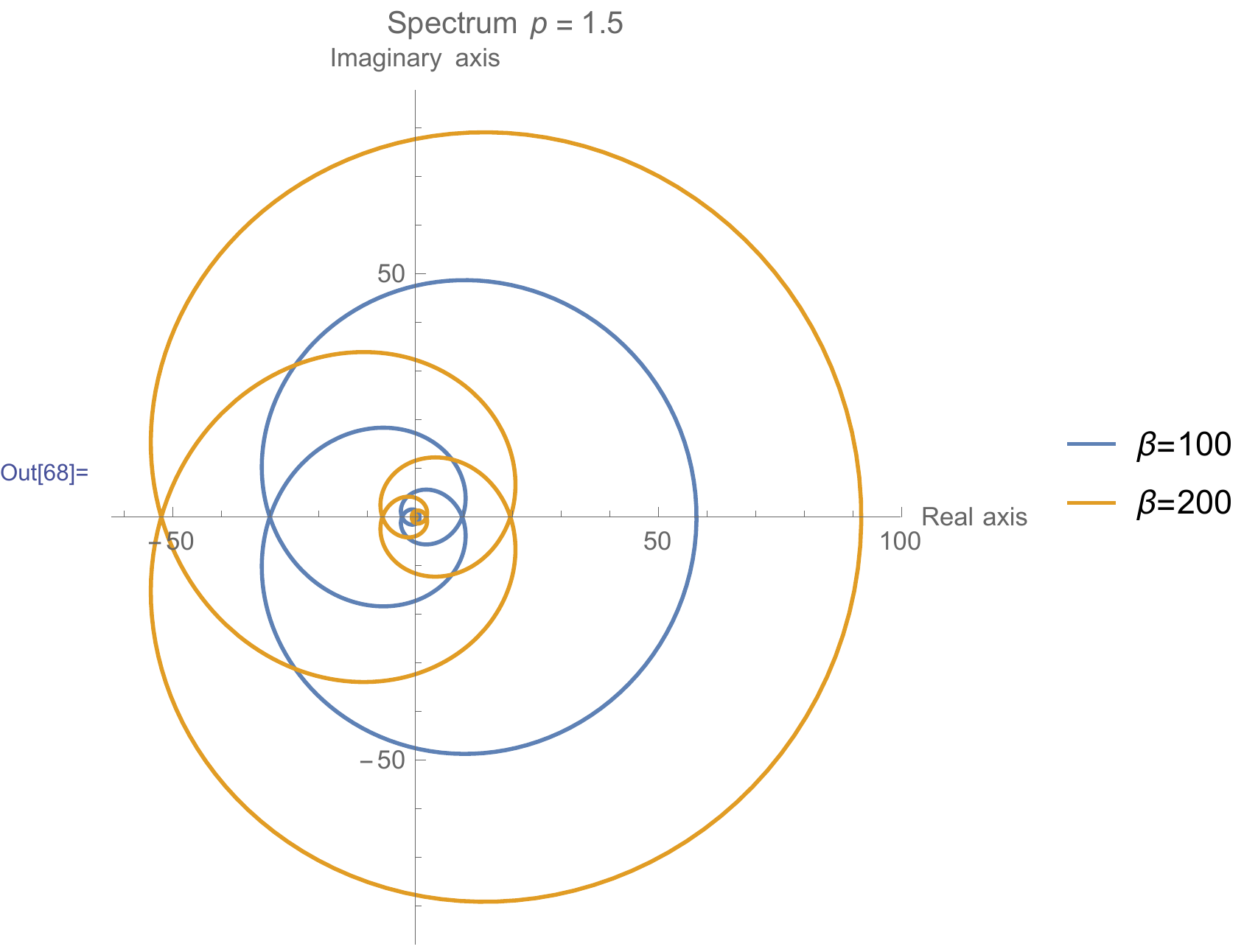}
\end{figure}

\vskip 2.0cm

In the Figure 3, we consider the particular case $p=\infty,$ with $\beta=1,10,100,1000,10000.$ First, observe that for all $\beta>0,$ $\Gamma_{\infty,\beta}$ cuts to the real axis in the point $1$ ($t=0$). In addition, this case seems to have a very special particularity; the envelope to the family of curves $\{\Gamma_{p,\beta}\}_{\beta>0}$ is the unit circle.


\begin{figure}[h]
\caption{}
\centering
\includegraphics[width=0.6\textwidth]{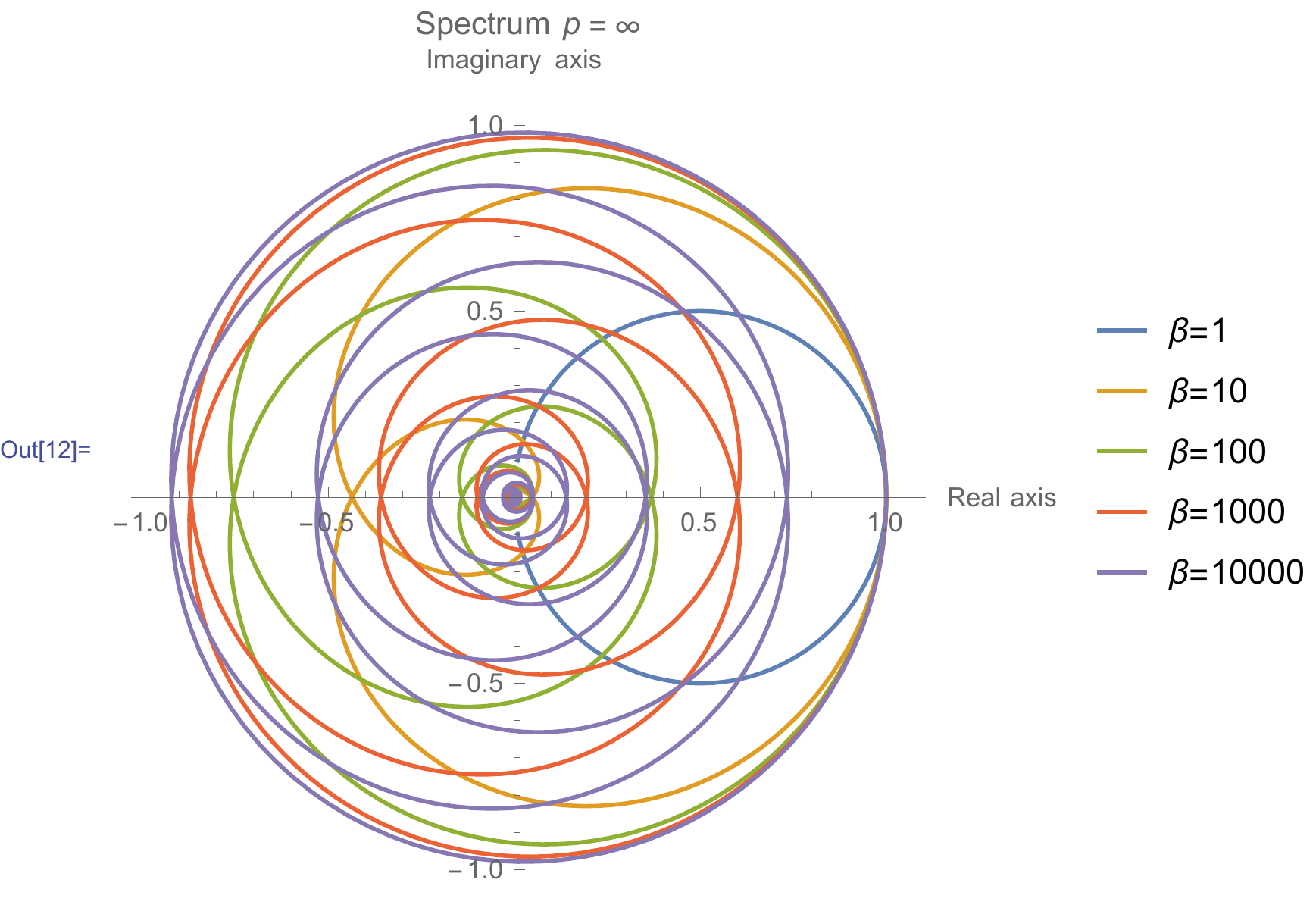}
\end{figure}

So, we establish the following open questions:\begin{itemize}
\item[(i)] Let $1<p<\infty$ and $0<\beta\leq 1.$ The curves $\Gamma_{p,\beta}$ do not cut the imaginary axis, and only cut the real axis in $0$ and $\frac{\Gamma(\beta+1)\Gamma(1-1/p)}{\Gamma(\beta+1-1/p)}.$

\item[(ii)] Let $1<p<\infty$ and $\beta>1.$ Study how many times the curves $\Gamma_{p,\beta}$ cut the real and imaginary axes and where they do it.

\item[(iii)] For $p=\infty,$ study if the envelope of the family of curves $\{\Gamma_{p,\beta}\}_{\beta>0}$ is the unit circle.

\end{itemize}

Now, we show some pictures about the cases with $\Real\beta>0$ and $\Imag\beta\neq 0.$ In these cases, the behaviour of $\Gamma_{p,\beta}$ is more difficult to predict. However, we will observe particularities and nice curves.

The Figure 4 shows that $\Gamma_{2,1-i}$ and $\Gamma_{2,1-i}$ are symmetrical with respect to the real axis. In fact, it is easy to see that $\Gamma_{p,\beta}$ and $\Gamma_{p,\overline\beta}$ are symmetrical.

\begin{figure}[h]
\caption{}
\centering
\includegraphics[width=0.6\textwidth]{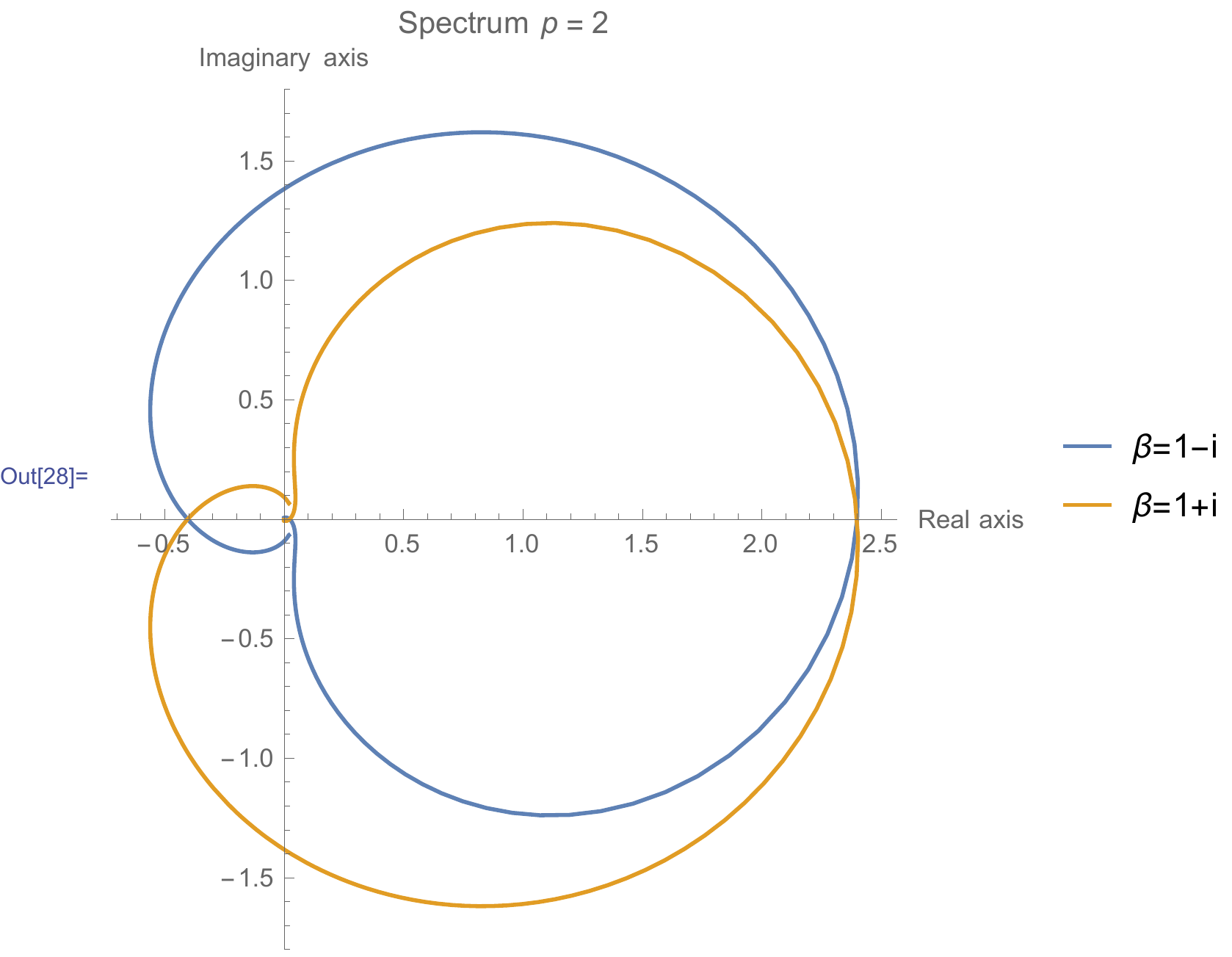}
\end{figure}

The Figure 5 shows that $\Gamma_{1.5,100+i}$ and $\Gamma_{1.5,100+20i}$ are very close to $\Gamma_{1.5,100}$ (compare to Figure 3), and $\Gamma_{1.5,100+100i}$ is considerably different.

\begin{figure}[h]
\caption{}
\centering
\includegraphics[width=0.6\textwidth]{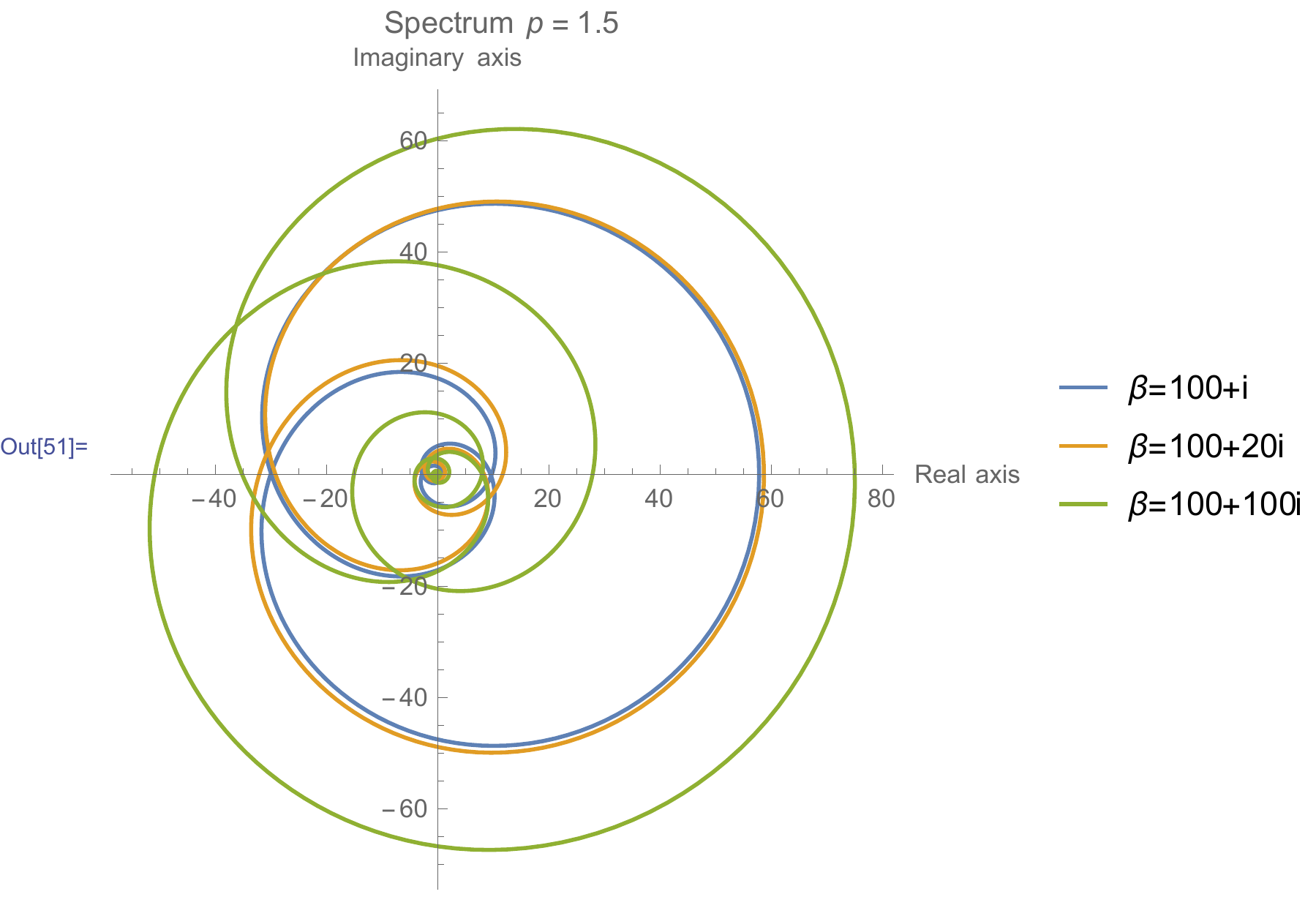}
\end{figure}

The Figure 6 represents a case where the gap between the real and imaginary part of $\beta$ is large. This allows us to observe that we are getting spirals centered at the origin.

\begin{figure}[h]
\caption{}
\centering
\includegraphics[width=0.6\textwidth]{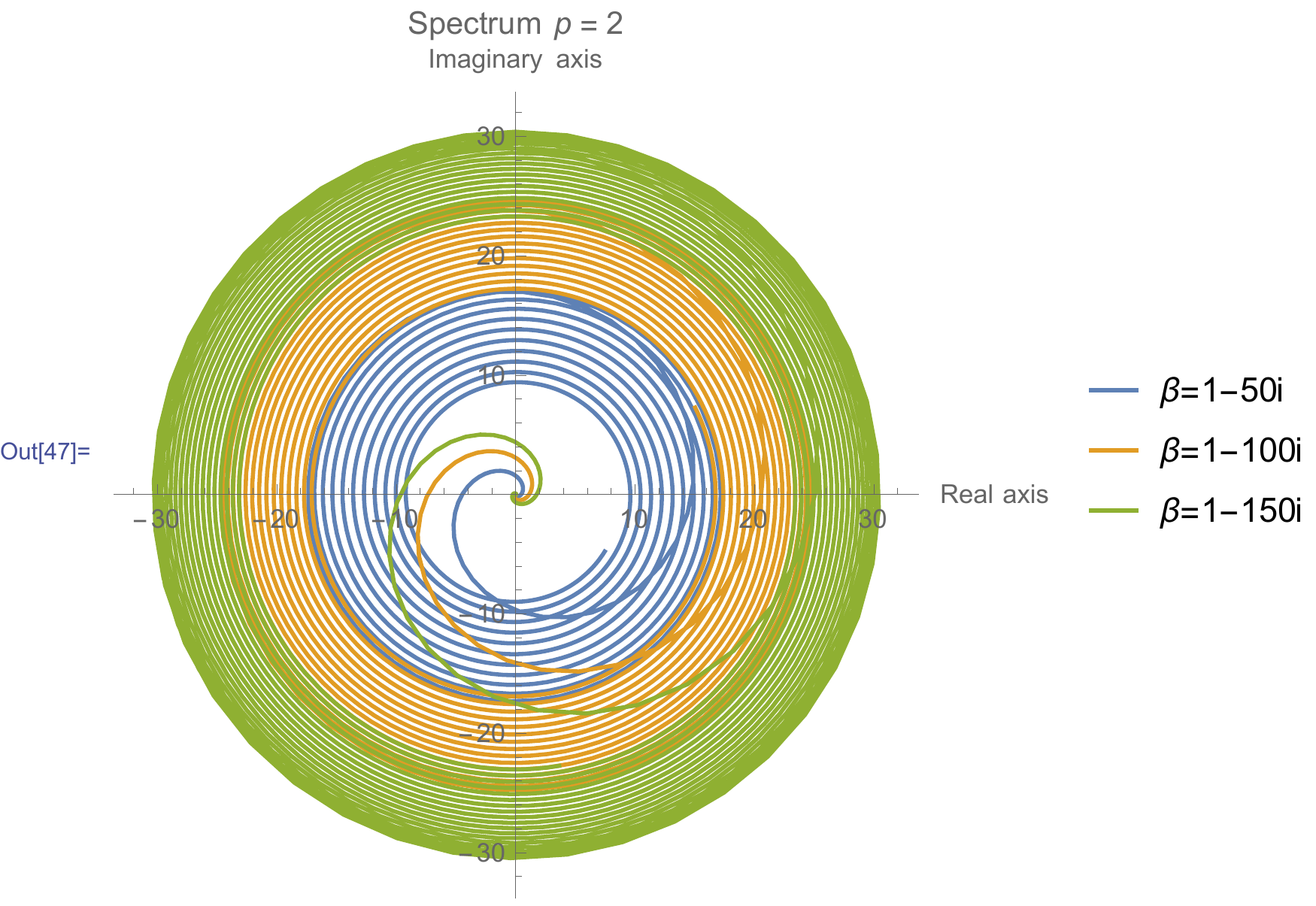}
\end{figure}

\vskip 5.0cm

For $p=\infty $ and large values of $|\beta|,$ the Figure 7 shows how similar variations on the growing of the real and imaginary parts of $\beta$ allow to control the curves $\Gamma_{\infty,\beta}$ in similar regions of the complex plane. However, the same variations only on the imaginary part imply large variations on the region where the curves are.

\begin{figure}[h]
\caption{}
\centering
\includegraphics[width=0.6\textwidth]{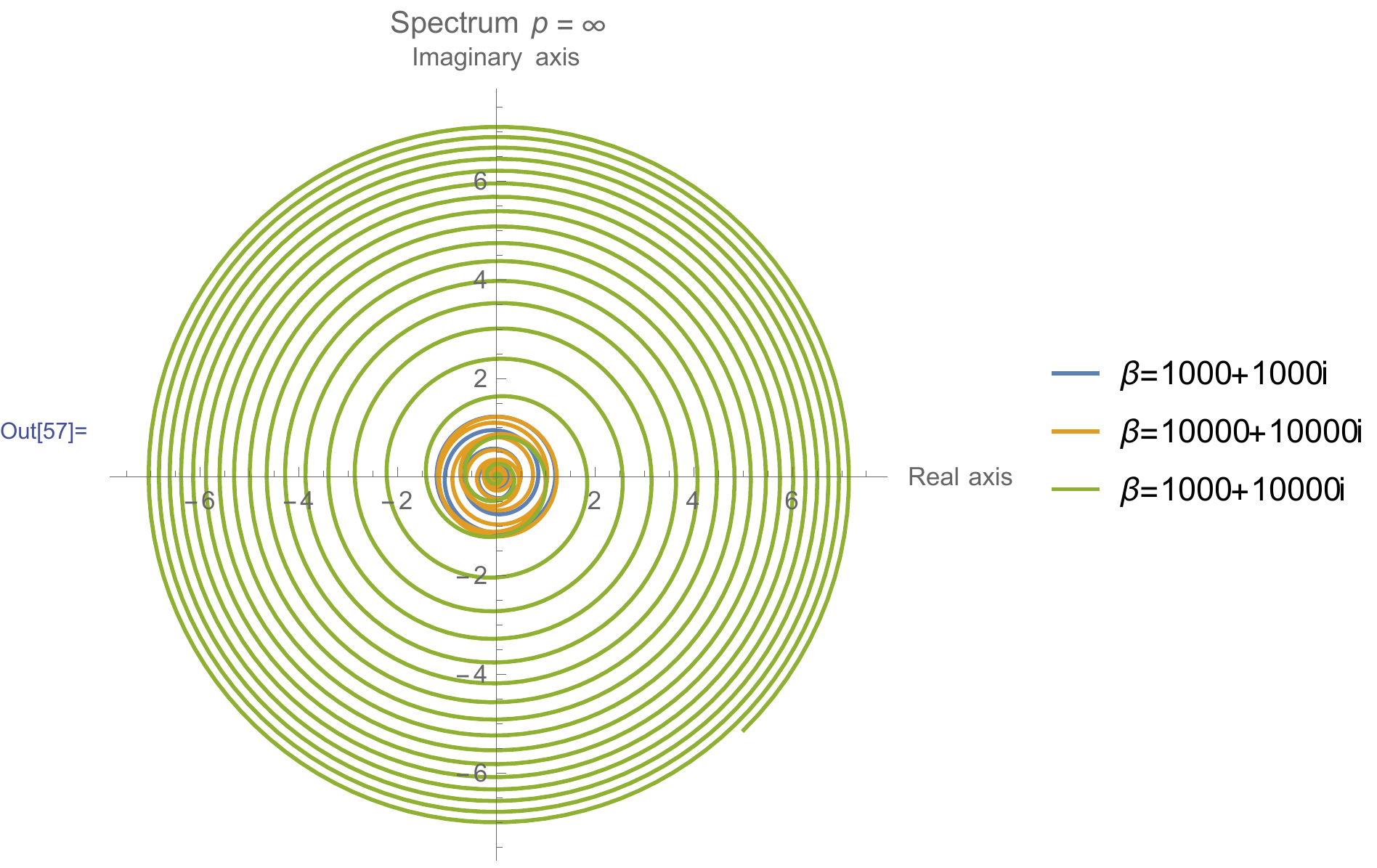}
\end{figure}

\vskip 1.0cm
\subsection*{\it Acknowledgments.} The authors wish to thank J.E. Gal\'{e}, L. Roncal, J. L. Varona and M. P. Velasco for the pieces of advice
and assistance provided in order to obtain some of the above results.


\begin{thebibliography}{999}
\bibitem{A} L. Abadias. {\it A Katznelson-Tzafriri theorem for Ces\`{a}ro bounded operators.} Studia Math. {\bf 234} (2016), no. 1, 59–82.

\bibitem{ADT} L. Abadias, M. De Le\'on, J. L. Torrea. {\it Non-local fractional derivatives. Discrete and continuous} J. Math. Anal. Appl. {\bf 449} (2017), no. 1, 734–755.

\bibitem{ALMV} L. Abadias, C. Lizama, P. J. Miana and M. P. Velasco. {\it Ces\`{a}ro sums and algebra homorphisms of  bounded operators.} Israel J. Math. {\bf 216} (2016), no. 1, 471–505.

\bibitem{Bonet} A. Albanese, J. Bonet and W. J. Ricker. {\it Spectrum and compactness of the Ces\`aro operator on weighted $\ell_p$ spaces.} J. Aust. Math. Soc. {\bf 99} (2015), no. 3, 287–314.


%
%



\bibitem{arvanti} A.G. Arvanitidis and A.G. Siskakis. {\it Ces\`{a}ro operators on the Hardy spaces of the half-plane.} Canad. Math. Bull. {\bf 56} (2013), 229--240.

\bibitem{AtEl09} F. M. Atici and P. W. Eloe. {\it Initial value problems in discrete fractional calculus.} Proc. Amer. Math. Soc.
{\bf 137} (3), (2009), 981-989.

\bibitem{AtSe10} F.M. Atici and S. Seng\"ul. {\it Modeling with fractional difference equations.} J. Math. Anal. Appl. {\bf 369} (2010), 1-9.


\bibitem{Brown} A. Brown, P. R. Halmos and A. L. Shields. {\it Ces\`{a}ro operators.} Acta Sci. Math. {\bf 26} (1965),
125--137.

\bibitem{Lizama1} S. Calzadillas, C. Lizama and J. G. Mesquita. {\it A unified approach to discrete fractional calculus and applications.} Preprint, 2014.

\bibitem{Cesaro} E. Ces\`{a}ro. {\it Sur la multiplication des series.} Bull. Sci. Math. {\bf 14} (2) (1890), 114--120.







%
%
%


\bibitem{Cowen} C. C. Cowen. {\it Subnormality of the Ces\`{a}ro operator and a semigroup of composition operators.}
Indiana Univ. Math. J. {\bf 33} (1984), 305–-318.


\bibitem{curbera} G. Curbera and W. Ricker. {\it Spectrum of the Ces\`{a}ro operator in $\ell^p$.} Arch. Math. {\bf 100} (2013), 267–-271.

\bibitem{Nagel} K. J. Engel and R. Nagel. {\it One-Parameter Semigroups for Linear Evolution Equations.} Graduate texts in Mathematics. {\bf 194}, Springer, New York, 2000.

\bibitem{ET} A. Erd\'elyi, F. G. Tricomi. {\it The aymptotic expansion of a ratio of Gamma functions.} Pacific J. Math. {\bf 1} (1951), 133-142.

%



\bibitem{Gale} J. E. Gal\'e and A. Wawrzy\'nczyk. {\it Standard ideals in weighted algebras of Korenblyum and Wiener types.} Math. Scand. {\bf 108} (2011), 291--319.



\bibitem{Go12} C. S. Goodrich. {\it On a first-order semipositone discrete fractional boundary value problem.} Arch. Math. {\bf 99} (2012), 509-518.


\bibitem{Haase} M. Haase. {\it The functional calculus for sectorial operators.} Operator Theory: Advances and Applications. {\bf 169}, Birkh\"auser Verlag, Basel, 2006.

\bibitem{HLP} G. H. Hardy, J. E. Littlewood, G. Polya. {\it Inequalities}. Cambridge University Press, 1934.



%

\bibitem{KMP} A. Kufner, L. Maligranda and L-E. Persson. {\it The prehistory of the Hardy inequality.} Amer. Math. Monthly. {\bf 113} (8) (2006), 715--732.




\bibitem{Leibo} G. Leibowitz. {\it  Spectra of discrete Ces\`{a}ro operators.} Tamkang J. Math. {\bf 3} (1972), 123-–132.




\bibitem{Lizama} C. Lizama. {\it The Poisson distribution, abstract fractional difference equations, and stability.} To appear in Proceedings of the American Mathematical Society.

\bibitem{LMPL} C. Lizama, P. J. Miana, R. Ponce, L. S\'anchez-Lajusticia. {\it On the boundedness of generalized Ces\`{a}ro operators on Sobolev spaces.} J. Math. Anal. Appl. {\bf 419} (2014), no. 1, 373--394.





\bibitem{PLMSSis} A.G. Siskakis. {\it Composition semigroups and the Ces\'{a}ro operator on $H^p$}. J. London Math.
Soc. {\bf 36} (2) (1987), 153--164.


\bibitem{ArchivSis} A. Siskakis. {\it On the Bergman space norm of the Ces\`{a}ro operator.} Arch. Math. {\bf 67} (1996),
312–-318.

\bibitem{Sis} A. Siskakis. {\it Semigroups of composition operators on spaces of analytic functions, a review.} Studies on composition operators (Laramie, WY, 1996), 229-252, Contemp. Math., 213, Amer. Math. Soc., Providence, RI, 1998.





\bibitem{Stempak} K. Stempak. {\it Ces\`{a}ro averaging operators.}
Proc. Roy. Soc. Edinburgh Sect. A {\bf 124} (1994), no. 1, 121--126.


\bibitem{Stevik} S. Stevi´c. {\it Ces\`{a}ro averaging operators.}
Math. Nachr. {\bf 248-249} (2003) 185--189.

%


\bibitem{Xiao} J. Xiao. {\it Ces\`{a}ro-type operators on Hardy, BMOA and Bloch spaces.}
Arch. Math. {\bf 68} (1997) 398--406.

\bibitem{Zygmund} A. Zygmund. {\it Trigonometric Series.} 2nd ed. Vols. I, II, Cambridge University
Press, New York, 1959.

\end{thebibliography}
\end{document}